\documentclass[a4,12pt]{article}%
\usepackage{amsmath}%
\usepackage{tikz}
\usepackage{amsfonts}%
\usepackage{amssymb}%
\usepackage{graphicx}
\newtheorem{theorem}{Theorem}

\newtheorem{assumption}{Assumption}

\newtheorem{definition}{Definition}
\newtheorem{example}{Example}

\newtheorem{lemma}{Lemma}

\newtheorem{proposition}{Proposition}

\newcommand{\calP}{{\cal P}}
\newcommand{\calM}{{\cal M}}
\newcommand{\ep}{\varepsilon}
\newcommand{\dN}{{{\bf N}}}
\newcommand{\dR}{{{\bf R}}}
\newcommand{\E}{{{\bf E}}}

\newcommand{\prob}{{{\bf P}}}

\newcommand{\calN}{\mathcal{N}}
\newcommand{\calS}{\mathcal{S}}
\newcommand{\calU}{\mathcal{U}}

\newenvironment{proof}[1][Proof]{\textbf{#1.} }{\ \rule{0.5em}{0.5em}}

\renewcommand{\baselinestretch}{1.20}

\newcounter{figurecounter}
\newcommand{\myendexample}{$\blacklozenge$}

\newcommand{\numbercellong}[1]
{
\begin{picture}(40,20)(0,0)
\put(0,0){\framebox(40,20)} \put(20,10){\makebox(0,0){#1}}
\end{picture}
}

\begin{document}

\title{Dynamic Sender-Receiver Games%
\thanks{The research of Solan and Renault and Vieille
were supported by the Israel Science Foundation (grant number 212/09) and the Agence  Nationale de la Recherche (grant ANR-10-BLAN 0112).}}

\author{J\'er\^ome Renault\thanks{TSE (GREMAQ, Universit\' e Toulouse 1),
  21 all\' ee de Brienne, 31000 Toulouse, France. E-mail: \textsf{jerome.renault@tse-fr.eu}.},
Eilon Solan\thanks{School of Mathematical Sciences, Tel Aviv University, Tel
Aviv 69978, Israel. E-mail: \textsf{eilons@post.tau.ac.il}.},
and Nicolas Vieille\thanks{Departement Economics and Decision Sciences, HEC Paris, 1, rue de
la Lib\'{e}ration, 78 351 Jouy-en-Josas, France. E-mail: \textsf{vieille@hec.fr}.}}

\date{\today}

\maketitle

\begin{abstract}
We consider a dynamic version of sender-receiver games, where the
sequence of states follows an irreducible Markov chain observed by the sender.
Under mild assumptions, we provide a simple characterization of  the limit set of equilibrium
payoffs, as players become very patient. 
Under these assumptions, the limit set  depends on the Markov chain only through its invariant measure. 
The (limit) equilibrium payoffs are the feasible payoffs that satisfy an individual rationality condition for the receiver, and an incentive compatibility condition for the sender.
\end{abstract}

\section{Introduction}

Since Crawford and Sobel (1982), sender-receiver games, or
cheap-talk games, have become a natural framework for studying
issues of information transmission between a privately informed
`expert' and an uninformed decision maker, where the two parties
have non-aligned interests.

When the decision maker acts only \emph{once}, the   extent to which
information can be shared at equilibrium has been studied
extensively, when `talk' takes place prior to the decision stage.
While Crawford and Sobel (1982), see also Green and Stokey (2007),
have focused on the case where communication is limited to a single
costless and non-verifiable message from the sender to the receiver,
more recent papers have shown that this restriction is not
innocuous, and have characterized the equilibrium outcomes for
general cheap-talk games, see Krishna and Morgan (2001), Aumann and
Hart (2003).\footnote{The case of verifiable messages has also been
studied in detail, see Forges and Koessler (2008).} This work has
been motivated by numerous concrete situations. We refer to Krishna
and Morgan (2008), Farrell and Rabin (1996), and Sobel (2009) for a
discussion of these applications.


The present work is motivated by the following observation. Whether
the sender is a financial advisor who provides advice to a client,
an expert who is consulted on a project, or a referee on a
project/person, the situation often calls for a dynamic approach.
Indeed, the financial advisor provides advice on a series of
investments, and the expert and the referee may be consulted on
successive, related projects.

Golosov, Skreta, Tsyvinsky and Wilson (2009) consider such a
situation. They assume that the sender repeatedly sends messages,
the receiver repeatedly makes decisions, while the state of the
world remains \emph{fixed} throughout. Within the Crawford and Sobel
framework (continuum of states/messages), they show that, for some 
specifications on the initial distribution on states,
(necessarily complex) equilibria exist, that achieve full
revelation of the state of the world in finite time.

We here deal with situations in which the state of the world may
change through time. Specifically, we assume that the successive
states form an irreducible Markov chain over some finite set. In
every stage, the sender issues a message/recommendation, and the
receiver makes a decision. States are only known to the sender, and
payoffs only depend on the current state and on the receiver's decision, but not on the message sent by the sender.

Since states are autocorrelated, any information disclosed in stage
$n$ provides valuable information in later stages as well, as in
Golosov et al. (2009). Yet, since the Markov Chain is irreducible,
this information becomes eventually valueless.

Intuitively,  the inter-temporal situation puts some restrictions on
the players' behavior. As an illustration, the opinion of an expert
who systematically provides laudative reports will eventually come
to be discounted, if not ignored, since the decision maker is aware
of the fact that the time-average report of the quality of
people/projects should reflect the invariant measure of the states of the world.
On the other hand, an expert who genuinely provides accurate
information to promote efficiency, but sees that the decision maker
only acts in his interests, may become wary and may stop to provide
valuable information to the decision maker. As is well-known from
repeated games, the sender may indeed provide  powerful incentives
by conditioning his future communication policy on the behavior of
the decision maker. Similar insights already appear in the literature on dynamic contracting, 
see Baron and Besanko (1984), Besanko (1985) or  Battaglini (2005).

Our paper relates to the recent and growing literature on incomplete
information games, in which the uncertainty evolves, see, e.g.,
Athey and Bagwell (2008),  Mailath and Samuelson (2001), Phelan (2006), Renault (2006), Wiseman (2008), and H\"orner, Rosenberg, Solan and Vieille (2010) and, especially, Escobar and Toikka (2010).

We provide a characterization of the limit
set of sequential equilibrium payoffs, when players are very
 patient.

Our main findings are the following. 
We first show (Theorem 1) that a feasible payoff
vector is a (limit) equilibrium payoff as soon as the following two conditions are met. On
the one hand, the payoff of the receiver should be at least his babbling equilibrium payoff. This
condition is an individual rationality condition. Indeed, the latter payoff is equal to the receiver  minmax payoff in the dynamic game since the receiver has the option
to ignore the announcements of the sender. On the other hand, the sender's payoff should satisfy 
an incentive compatibility condition, which reflects the fact that the sender has the option of 
substituting artificially generated states to the true ones when playing, as long as the artificial states are 
statistically undistinguishable from the true ones. As it turns out, this incentive constraint takes the
form of finitely many linear inequalities.

In the corresponding equilibria,  with high probability the sender truthfully reports
the current state  most of the time , while the receiver responds in a
stationary manner to the announcements of the sender, and checks
that the distribution of these announcements is consistent with the
invariant measure.\footnote{While this is reminiscent of the revelation principle, we must stress that no revelation 
principe applies in our setting.}

We next show (Theorem 2) that the converse inclusion holds under some additional condition on the Markov chain, which we call \textbf{Assumption A}: any limit equilibrium payoff must satisfy the individual rational condition and the specific version of the incentive compatibility requirement of Theorem 1. 

A noteworthy consequence is that, under \textbf{Assumption A},  the limit set of
equilibrium payoffs 
does not depend on \emph{how} successive states are correlated, nor
on fine details of the sequence of states, but only on the invariant
measure. It is also irrelevant whether the sender learns some, or even all, of the 
realization of the future states in advance. In particular, the set of equilibrium payoffs can be
computed \emph{as if} successive states were
independent. 


Our results are valid  for  a large (open) class of payoff functions for the static game, but not for all of them. More precisely, we prove that for generic payoff functions (and under {\bf Assumption A}),  either our results hold, or all equilibria of the repeated game are payoff-equivalent to   babbling equilibria.



The paper is organized as follows. The model is described in
Section \ref{section model}. In Section \ref{example2.1} we explain most insights by means of an example. The main results appear in Section \ref{section results}, together with an illustration. Proofs are discussed in Section 
\ref{section proof} and the Appendix. Additional results and comments are provided in Section \ref{sec_further}.
The Appendix contains all proof details.

\section{Model}
\label{section model}

We study dynamic sender-receiver games, in which the state of the
world changes through time. At each stage $n\geq 1$, the sender (player 1)
 observes the current state of the world
$s_n\in S$, and makes an announcement $a_n\in A$.
Upon observing $a_n$, the receiver (player 2) chooses an
action $b_n\in B$. The current action $b_n$,
together with the current state $s_n$, determines the
utility vector $u(s_n,b_n) \in
\dR^2$ at stage $n$. 
 Only
the action $b_n$ is then publicly disclosed. We thus maintain the assumption that payoffs are not observed.%
 The two players share a common discount factor
$\delta$.


We assume throughout that the set of states $S$, the set of
messages $A$, and the set of actions $B$, are finite. We also
assume that  there are at least as many messages as states. This assumption ensures that the only motives for concealing the state are
strategic. We thus leave aside  situations in which, due to capacity
constraints, the sender might be forced to choose which feature of
the state to reveal.
For simplicity, we will actually assume throughout that the set $A$ of messages coincides with the set $S$ of states. (As will be seen, this assumption is without loss of generality in our setup.)

We assume that the states $(s_n)$ follow a Markov chain over
$S$, with transition function $p(\cdot\mid\cdot)$,
which is irreducible and aperiodic.%
\footnote{That is, for any two states $s,t\in S$, and for every $N
\in \dN$ large enough, the probability of moving from $s$ to $t$ in
exactly $N$ stages is positive.} The Markov chain therefore admits a
unique invariant measure, $m\in  \Delta(S)$. For convenience, we
assume that the first state, $s_1$, is drawn according to
$m$. This ensures that the law of $s_n$ is equal to $m$,
for every $n \geq 1$. 

\bigskip
In this setup, a strategy  of the sender maps past and current realized states, and past play, into a mixed message,
and is thus a map  $\sigma:\cup_{n\geq 0}
(S\times A\times B)^n\times S\to \Delta(A)$,
while a strategy of the receiver is a map $\tau:\cup_{n\geq 0}(A\times B)^n\to \Delta(B)$.
A \emph{stationary} strategy of the receiver is a map $y : A\to \Delta(B)$, with the interpretation that the receiver chooses his action according to
$y(\cdot \mid a)\in \Delta(B)$ whenever told $a\in A$.

\bigskip
Our goal is to study to what extent  the dynamic structure of the
game affects the equilibrium outcomes. Formally, we aim at providing
a characterization of the limit set of sequential equilibrium
payoffs,
and at understanding equilibrium behavior, when players are very patient.%

\section{An Example}
\label{example2.1}

We here illustrate our main results by means of a simple example. There
are two states, $S = \{L,R\}$, and two actions for the receiver, $l$
and $r$. Successive states are independent and equally likely. Payoffs
are given by the two tables in Figure \arabic{figurecounter},
where $c$ is a fixed parameter, with $c\in (1,2)$. The sender and the receiver are respectively players 1 and 2. 

\centerline{
\begin{picture}(90,60)(-10,-20)
\put( 20,27){$l$}
\put( 60,27){$r$}
\put( 0,0){\numbercellong{$c,2$}}
\put(40,0){\numbercellong{$2,1$}}
\put(20,-15){State $L$}
\end{picture}
\ \ \ \ \ \ \ \ \ \ \ \ \ \ \ \ \ \ \ \
\begin{picture}(90,60)(-10,-20)
\put( 20,27){$l$}
\put( 60,27){$r$}
\put( 0,0){\numbercellong{$1,-1$}}
\put(40,0){\numbercellong{$2,1$}}
\put(20,-15){State $R$}
\end{picture}
} \centerline{Figure \arabic{figurecounter}: The payoffs of the
two players.}

\addtocounter{figurecounter}{1}

The one-shot information transmission game has a unique equilibrium, in which the receiver plays $r$ with probability 1.
To see this, note that the sender strictly prefers action $r$ over action $l$, no matter what the state is. Thus, at equilibrium, all messages that are sent with positive probability induce the same mixed action by the Receiver. This constant mixed action, being always \textit{ex post} optimal for the Receiver, is therefore also \textit{ex ante} optimal. It must thus assign probability one to action $r$.
\bigskip

All equilibria in the one-shot game are therefore babbling
equilibria.\footnote{In the sense that the action of the receiver
is independent of the message sent by the sender.} Plainly, the
dynamic game admits a babbling equilibrium, in which the sender repeatedly makes the same announcement,  the receiver
treats the announcements as being non-informative, and
plays $r$ in every stage. On the other hand, the receiver can always choose to ignore
the announcements of the sender, and to play $r$ in every stage,
thereby getting 1. As a result, the babbling equilibrium is the
\emph{worst} equilibrium for the receiver, in both the one-shot
and in the dynamic game.

\bigskip

We claim that the dynamic game has equilibrium payoffs that are arbitrarily close to $(\frac{2+c}{2},\frac32)$.
In particular, and in contrast with the receiver,  there are  equilibrium payoffs for the sender that are below the babbling equilibrium payoff. 
Here is the intuition. The sender announces the true state at every stage.
The receiver listens to the announcements of the sender, and plays $l$ when told $L$, and $r$ when told $R$.
 To prevent  the sender from announcing $R$ in every stage, the receiver monitors the announcements of the sender, and stops listening if there is an obvious bias (towards either $L$ or $R$).
Under the constraint that he should
announce both states equally often, the expected payoff of the sender is highest when he reports truthfully.

While this intuition is simple, formalizing it into an equilibrium of the discounted game is not straightforward.
Indeed, because payoffs are discounted, the sender may have a preference to send at first the message $R$ more frequently.

\bigskip

We start with a simple construction that yields an equilibrium payoff distinct from $(2,1)$.
Assume that the discount factor  satisfies $\delta > \displaystyle \frac{4-2c}{3-c}$, and consider the following strategy profile.
\begin{itemize}
\item   At odd stages, the sender announces truthfully the current
stage, and the receiver plays $l$ if told $L$, and $r$ if told
$R$. \item   At even stages, the sender announces a constant
message, and the receiver plays the action that he did \emph{not}
play in the previous stage. \item   If the receiver deviates, both
players switch to the babbling equilibrium forever.
\end{itemize}
Under this strategy profile,
 expected payoffs are equal to $\displaystyle \frac{1}{1+\delta}\left( \frac{2+c}{2}+\delta\frac{5+c}{4}\right)$ and
$\displaystyle \frac{1}{1+\delta}\left( \frac{3}{2}+ \frac{3}{4}\delta\right)$ respectively.
Because a deviation of the receiver is followed by the babbling equilibrium,
which yields 1 to the receiver, and because the (conditional) expected payoff of the receiver is at least 1 in every stage,
no deviation of the receiver is profitable.
Regarding the sender, it is sufficient to show that he cannot profit by deviating in any block of two stages.
In such a block the sender has two possible deviations:
to announce $L$ in the first stage of the block when the true state is $R$,
and to announce $R$  in the first stage of the block when the true state is $L$.
In the former case, he gets 1 at the first stage and 2 in the second
(instead of 2 at the first stage and $\frac{1+c}{2}$ at the second stage if he announces truthfully).
In the latter case, he gets 2 at the first stage and $\frac{1+c}{2}$ at the second stage
(instead of $c$ at the first stage and $2$ at the second stage if he announces truthfully).
The choice of $\delta$ ensures that none of these deviations is profitable.

\bigskip

To get payoffs closer to $(\frac{2+c}{2},\frac32)$, we will be
relying on a slightly more complex construction. We let the size
$2N$ of a block be large enough, so that a law of large numbers
will apply. Once $N$ is fixed, we let the discount factor
$\delta$ be high enough, so that the contribution of  any
individual block to the overall discounted payoff is very small.

We first describe a pure strategy $\tau$ of the receiver.
In each block (unless if the receiver has deviated earlier),
the receiver listens to the sender's announcements, plays $l$ if told $L$, $r$ if told $R$,
until the number of announcements of either $L$ or $R$ exceeds $N$.
When this is the case, the receiver stops listening to the sender's announcements,
and repeats the least frequent action until the end of the current block.%
\footnote{An alternative construction, that we adopt in the
general case, is for the receiver to generate a specific sequence
of fictitious announcements, and continue as if the sender's
announcements were equal to the fictitious ones.} In a sense,  the sender is restricted to announcing both states equally often in any given block of $2N$ stages. As such, the intuition here is similar to some extent to the one behind the \textit{linking mechanism} of Jackson and Sonnenschein (2007) and, even more, to the analysis in Escobar and Toikka (2010).\footnote{The present analysis  and the one in Escobar and Toikka (2010) were developed independently.}

If indeed the
sender reports truthfully the current state, there is a high
probability that the receiver will be listening  to the sender most of the time, and the expected payoff is
therefore close to $(\frac{2+c}{2},\frac32)$.

In contrast with the situation examined above,
it need not be optimal for the sender to report truthfully when facing $\tau$.
However, a crucial insight is that \emph{any} best reply of the sender to $\tau$ must be reporting truthfully most of the time, with high probability.
To see why, observe that any best reply achieves a payoff of at least,
say, $\frac{2+c}{2}-\ep$.
But since the receiver plays both actions $l$ and $r$ equally likely on each block,
this implies that with high probability the action of the receiver matches the state,
most of the time.

We let $\sigma$ be any pure best-reply of the sender to $\tau$.
On the equilibrium path, we let players play according to $\sigma$ and $\tau$.
By construction, the equilibrium property holds for the sender.
To deter the receiver from deviating,
both players switch forever to the babbling equilibrium once a deviation of the receiver is detected.
Since blocks are short, the expected continuation payoff of the receiver is close to $\frac32$ following any history,
while the receiver gets a payoff of 1 (or close to 1) if he deviates.

\section{Main Results}

\label{section results}

We here state and discuss two results on the limit set of equilibrium payoffs. Loosely speaking, according to Theorem \ref{theorem1}, 
all payoff vectors that are  individually rational for the receiver and incentive compatible for the receiver are (asymptotically) equilibrium payoffs. 
Theorem \ref{theorem2} proves the converse inclusion.  Further results are provided in Section \ref{sec_further}.

\subsection{Theorem \ref{theorem1}}

We start with some notations.
We denote by $\calM \subset
\Delta(S\times A)$ the set of {\em copulas} based on $m$; that is,
the set of distributions $\mu$ over $S\times A$
 whose marginals on $S$ and on $A$ are both equal to $m$.\footnote{Recall that the set $A$ of messages is a copy of $S$.}
 The set ${\cal M}$ is defined by a finite number of linear inequalities, hence it is a compact convex polyhedron,
 so it has finitely many extreme points.

We denote by $\mu_0\in \calM$ the specific distribution defined as
$\mu_0(s,s)=m(s)$ for each $s\in S$, and $\mu_0(s,a)=0$ if $s\neq
a$. Under $\mu_0$, the messages and the  states coincide
\emph{a.s}. Thus, the distribution  $\mu_0$ is the long-run
average distribution of the sequence
$(s_n,a_n)_n$ when the sender reports
truthfully the current state.

Given a copula $\mu\in\calM$, and a stationary strategy $y:A\to \Delta(B)$, we set
\[U(\mu,y):=\sum_{s\in S,a\in A}\mu(s,a)u(s,y(\cdot\mid a))\in \dR^2.\]
This is the expected payoff vector when the sender's report is drawn according to  $\mu(\cdot\mid s)$,
and the receiver plays $y$.

We denote by \begin{equation}
\label{equ minmax}
v^2 := \max_{b \in B} \sum_{s \in S} m(s) u^2(s,b)
\end{equation}
 the babbling equilibrium payoff for the receiver.
\begin{definition}\label{defE}
We let $E(\calM)$ denote the set of payoff vectors $U(\mu_0,y)$, where $y:A\to \Delta(B)$, that satisfy
\begin{description}
\item[C1.] $U^1(\mu_0,y)\geq U^1(\mu,y)$ for every $\mu\in \calM$.
\item[C2.] $U^2(\mu_0,y)\geq v^2$, 
\end{description}
\end{definition}

We define $\widehat{E}(\calM)$ as the set of payoff vectors
$U(\mu_0,y)\in E(\calM)$ where the inequalities in \textbf{C1} and
\textbf{C2} are strict. That is, $\widehat{E}(\calM)$ is the set
of vectors $U(\mu_0,y)$, $y:A\to \Delta(B))$, such that
\begin{description}
\item[D1.] $U^1(\mu_0,y)> U^1(\mu,y)$ for every $\mu\in \calM$,
$\mu\neq \mu_0$. \item[D2.] $U^2(\mu_0,y)> v^2$.
\end{description}

Note that condition \textbf{D1} holds as soon as the inequality
$U^1(\mu_0,y)> U^1(\mu,y)$ is satisfied for each of the finitely
many extreme points $\mu\neq \mu_0 $ of $\calM$.

We denote by  $SE_\delta$ the set of sequential equilibrium payoffs of the game with discount factor $\delta$.
Our first main result, Theorem \ref{theorem1}, shows that all payoffs in $E(\calM)$ can be obtained as equilibrium payoffs, provided that players are sufficiently patient. 

\begin{theorem}\label{theorem1}
Suppose that there exists a public randomizing device, which
outputs a (uniformly distributed) number in $[0,1]$ in every
stage, after the announcement of the sender. If
$\widehat{E}(\calM)\neq \emptyset$ then
\[E(\calM)\subseteq \lim\inf_{\delta\to 1} SE_\delta.\]
\end{theorem}

Theorem \ref{theorem1} means that for every $\gamma\in E(\calM)$ and for every $\ep>0$, there exists $\delta_0<1$ such that, for every $\delta \geq \delta_0$,
 the $\delta$-discounted game has a sequential equibrium payoff  within $\ep$ of $\gamma$. 
\footnote{We will actually prove the stronger statement that $\delta_0$ can be chosen to be independent of $\gamma$:
 \[\lim_{\delta\to 1} \sup_{\gamma\in E(\calM)}d(\gamma,SE_\delta)= 0.\]} The proof of Theorem \ref{theorem1} is provided in Section \ref{sec_prooftheorem1}.
\bigskip

Few comments are in order. 

The babbling payoff $v^2$ is  equal to the $\min\max$ value of the receiver in
the dynamic game. Hence condition \textbf{C2} in Definition \ref{defE} reads as an  individual rationality condition.
Condition \textbf{C1} is akin to an incentive compatibility
condition: under the constraint that the distribution of messages is equal to the distribution of states, 
truth-telling is optimal for the sender. According to Theorem \ref{theorem1}, any payoff vector $U(\mu_0,y)$ that is incentive compatible for the sender, and 
individually rational for the receiver, is an equilibrium payoff for $\delta $ large. 

Our construction will have the
somewhat surprising feature that the sender reports truthfully, at
least most of the time and with high probability. A direct
intuition can be provided, that is reminiscent of the revelation
principle in mechanism design. Let an equilibrium $(\sigma,\tau)$
be given. Consider the strategy profile where the sender reports
truthfully, and the receiver first computes the message that the
strategy $\sigma$ would have sent, and next plays what $\tau$
would have played given this message. We argue loosely that this
new profile (when supplemented with threats) is an equilibrium.
The key to the argument is twofold. On the one hand, the sender
can check that the receiver does indeed play as prescribed, and
does not use the additional information provided by the knowledge
of the true state. On the other hand, the threat of switching to
the babbling equilibrium is effective because the knowledge of the
state at a given stage becomes eventually valueless in predicting
distant stages, because of the irreducibility property of the
sequence of states. However, we should stress that no revelation principle applies in our setup, 
and our equilibrium construction relies on the  threat that the sender will stop 
providing information following a deviation. 

\bigskip

Theorem \ref{theorem1} relies on two assumptions. 
The public randomizing device can easily be dispensed with, provided one slightly extends the  communication options offered to  the players. To be specific, assume that the players are allowed to exchange  simultaneous 'messages', after the sender has reported a state. Under such an assumption, players can implement \textit{jointly controlled lotteries} as in Aumann and Maschler (1995), which can substitute for the randomization device. Details are standard and omitted.  However, when instead communication is restricted to a single message sent by the sender, then the existence of a public randomizing device is not without loss of generality, see Section \ref{sec_device} for an example.

\bigskip
Theorem \ref{theorem1} also requires  $\widehat{E}(\calM)$ to be non-empty. This is similar to the non-empty interior type of conditions 
which appear in Folk Theorems. Yet, we must stress that our assumption is somewhat stronger, since $\widehat{E}(\calM)$ need not be equal to the relative interior of $E(\calM)$, and the condition that $\widehat{E}(\calM)\neq \emptyset$ is not generically satisfied. We provide a robust example where $\widehat{E}(\calM)= \emptyset$ and elaborate further on this issue in Section \ref{sec_generic}. 

\bigskip

It is not an easy task  to rely on Definition \ref{defE} to check whether  a given payoff vector $U(\mu,y)$ belongs to
$E(\calM)$.  Fortunately, it
turns out that conditions \textbf{C1} and \textbf{D1} are equivalent to much
simpler conditions.

\begin{lemma}  \label{lemmbasic}
Let $y:A\to \Delta(B)$ be given. Conditions \textbf{C1} and \textbf{D1} are respectively equivalent to conditions \textbf{C'1} and \textbf{D'1} below.
\begin{description}
\item[C'1] $\displaystyle\sum_{s\in S}u^1(s,y(\cdot\mid s))\geq \sum_{s\in S}u^1(s,y(\cdot\mid\phi(s)))$, for every permutation $\phi$ over $S$.
\item[D'1] $\displaystyle\sum_{s\in S}u^1(s,y(\cdot\mid s))> \sum_{s\in S}u^1(s,y(\cdot\mid\phi(s)))$,
for every permutation $\phi$ over $S$ that is not the identity mapping.
\end{description}
\end{lemma}

The proof of Lemma \ref{lemmbasic} is in the Appendix. 
Interestingly, conditions \textbf{C'1} and \textbf{D'1} do
\emph{not} involve the invariant distribution $m$. The intuition is best explained in the case of two states, $s_0$ and $s_1$. Assume that the sender is 
considering mis-reporting the state, under the constraint that the distribution of reports matches the invariant distribution $m$ of the state. 
The only way to do this is to report $s_1$ instead of $s_0$, \emph{as often} as to report $s_0$ instead of $s_1$. Whether such a deviation is profitable 
is equivalent to asking how the \emph{unweighted} sum of the payoffs obtained in  $s_0$ when reporting $s_1$ and in $s_1$ when reporting $s_0$ compares
to the \emph{unweighted} sum of the payoffs obtained in the two states when reporting truthfully.

\subsection{Theorem \ref{theorem2}}

Our second main result provides the converse inclusion to that in Theorem \ref{theorem1}. 
It  requires
one substantive assumption on the behavior of the state, \textbf{Assumption A} below.

\begin{assumption}[Assumption A]
 There exist nonnegative numbers $\alpha_s$, $s\in S$,
with $\displaystyle \sum_{s\in S\setminus \{\bar s\}}\alpha_s\leq 1$ (for every $\bar s\in S$),
such that $p(s'\mid s)=\alpha_{s'}$ whenever $s'\neq s$.
\end{assumption}

\textbf{Assumption A} is restrictive. Yet it does \textit{e.g.} hold in the following cases. 

Assume first that changes
in the state are due to  shocks, which occur at random
times. Once drawn, the state remains constant until a
 shock occurs. The state is then drawn
anew, according to $m$. The inter-arrival times of the successive
shocks are i.i.d., and follow a geometric distribution. In that
case, \textbf{Assumption A} is met. Indeed, it suffices to set
$\alpha_t=\pi\times m(t)$ for every $t \in S$, where $\pi$ is the
per-stage probability of a shock. The parameter $\pi$ is here a
measure of the state persistence. When $\pi$ increases from 0 to 1,
the situation evolves from one in which the state remains constant
through time, to a situation in which successive states are
independent.

When $\pi=1$, the successive states are independent, and identically
distributed according to $m$. Thus, \textbf{Assumption A} holds in
the case of i.i.d. states.

\textbf{Assumption A} also holds in the benchmark case
where there are only two possible states.
 Indeed, denoting the two
states by $s_1$ and $s_2$, it suffices to set $\alpha_1 = p(s_1\mid
s_2)$ and $\alpha_2 = p(s_2\mid s_1)$. In particular, it is satisfied in the models in Athey and Bagwell (2008), Phelan (2006) and Wiseman (2008). 

As a further simple illustration, consider a symmetric random walk on three states. That is, whenever in a state, the chain moves to each of the two other states with probability $\frac12$. Again, \textbf{Assumption A} is met, with $\alpha_s=\frac12$ for each $s\in S$.

\bigskip

We denote by $NE_\delta$ the set of (Nash) equilibrium payoffs in the game with discount factor $\delta$.

\begin{theorem}\label{theorem2}
Suppose that \textbf{Assumption A} holds. Then, for every $\delta
<1$, one has
\[NE_\delta\subseteq E(\calM).\]
\end{theorem}

Provided that $\widehat{E}(\calM)\neq \emptyset$ and that \textbf{Assumption A} is met, Theorems \ref{theorem1} and \ref{theorem2}
thus imply that the set of sequential equilibrium payoffs $SE_\delta$  converges to the set $E(\calM)$ (as soon as a randomizing device is available).

Note that the set $\calM$ of copulas only depends on the invariant measure $m$, and not on finer details of the transition function.
A striking implication of the characterization is that, under \textbf{Assumption A}, the limit
set of equilibrium payoffs therefore only depends on the invariant measure $m$.
In particular, the limit set of equilibrium payoffs is the same as when the states are drawn independently across stages.
That is, the amount of state persistence is irrelevant for the determination of the limit set of equilibrium payoffs.

If the initial state $s$ were to remain fixed throughout the play,
the game would fall into the class of repeated games with
incomplete information introduced by Aumann and Maschler (1995).
(This is the setup studied in Golosov et al. (2009).) In this
case, the limit set of  discounted equilibrium payoffs, when $\delta$ goes to 1,   is typically not equal to
$E(\calM)$. Hence, there is a discontinuity in the limit set of
equilibrium payoffs when  successive
states become perfectly autocorrelated.\footnote{Our Theorems
\ref{theorem1} and \ref{theorem2} extend to cover the case of
uniform equilibrium payoffs.}

By contrast, for a fixed discount factor, the set of equilibrium
payoffs is upper hemi-continuous with respect to the transition
function. The source of this apparent paradox can be traced back
to the fact that, in loose terms, the convergence of the set
$SE_\delta$ to $E(\calM)$ is slowlier, the more correlated successive states
are.

\bigskip

The main insight to be derived from Theorem \ref{theorem2} is the following. 
The incentive compatibility condition 
\textbf{C2} is a very \emph{strong  }one. Indeed, it only  requires from deviations
that the distribution of  announcements matches the invariant measure. In particular, according to Theorem \ref{theorem2}, all equilibria are 
payoff equivalent to equilibria in which the receiver only checks that the announcements frequencies are consistent with $m$. 
Yet, much more sophisticated checks would be available to the receiver. The receiver might \textit{e.g.} check that the empirical distribution of two-letter words $(s,s')$ 
matches the transition function $p$,  as in Escobar and Toikka (2010), or look at the distribution of three-letter words, etc. This might potentially allow the receiver to impose \emph{weaker} incentive constraints than the one in \textbf{C1}, and therefore, allow for equilibrium payoffs outside of $E(\calM)$. Theorem \ref{theorem2} thus identifies one class 
of Markov chains for which this is not the case. 

 We provide below an example where
 the conclusion of Theorem \ref{theorem2} fails to hold if \textbf{Assumption A} is not satisfied.
Thus, in general, the limit set of sequential equilibrium payoffs does not only depend on the invariant measure, but also on
finer details of the transition function.

\begin{example}
\label{example7} Consider a game with 5 states $S:=\{0,1,2,3,4\}$.
The sequence of states follows a random walk on $S$. When in $s$,
the chain moves either to $s+1$ $(\rm{mod}\ 5)$ or to $s-1$
$(\rm{mod}\ 5)$ with equal probabilities. The action set $B$ of
the receiver coincides with $S$, and the payoff function is
described in Figure \arabic{figurecounter}, where $c>1$.

\centerline{
\begin{picture}(220,120)(-20,0)
\put( 10,107){$b=0$}
\put( 50,107){$b=1$}
\put(90,107){$b=2$}
\put(130,107){$b=3$}
\put(170,107){$b=4$}
\put(-30,87){$s=0$}
\put(-30,67){$s=1$}
\put(-30,47){$s=2$}
\put(-30,27){$s=3$}
\put(-30,7){$s=4$}
\put( 0,0){\numbercellong{$0,0$}}
\put(40,0){\numbercellong{$0,0$}}
\put(80,0){\numbercellong{$0,0$}}
\put(120,0){\numbercellong{$0,0$}}
\put(160,0){\numbercellong{$1,1$}}
\put( 0,20){\numbercellong{$0,0$}}
\put(40,20){\numbercellong{$0,0$}}
\put(80,20){\numbercellong{$0,0$}}
\put(120,20){\numbercellong{$1,1$}}
\put(160,20){\numbercellong{$0,0$}}
\put( 0,40){\numbercellong{$0,0$}}
\put(40,40){\numbercellong{$0,0$}}
\put(80,40){\numbercellong{$1,1$}}
\put(120,40){\numbercellong{$0,0$}}
\put(160,40){\numbercellong{$0,0$}}
\put( 0,60){\numbercellong{$c,0$}}
\put(40,60){\numbercellong{$1,1$}}
\put(80,60){\numbercellong{$0,0$}}
\put(120,60){\numbercellong{$0,0$}}
\put(160,60){\numbercellong{$0,0$}}
\put( 0,80){\numbercellong{$1,1$}}
\put(40,80){\numbercellong{$c,0$}}
\put(80,80){\numbercellong{$0,0$}}
\put(120,80){\numbercellong{$0,0$}}
\put(160,80){\numbercellong{$0,0$}}
\end{picture}
} \centerline{Figure \arabic{figurecounter}: The game in Example
\ref{example7}.}
\end{example} \addtocounter{figurecounter}{1}

Thus, both players receive a payoff 1 if the action matches
the current state, and 0 otherwise, except when the receiver
chooses action 1 in state 0, or action 0 in state 1.

The payoff vector $(1,1)$ is not in $E(\calM)$ as soon as $c > 1$.
Indeed, the stationary strategy $y:A\to \Delta(B)$ defined by
$y(s\mid s)=1$ is the only strategy such that $U(\mu_0,y)=(1,1)$.
But then, the sender profits by reporting $t=1$ whenever $s=0$,
and $t=0$ whenever $s=1$. On the other hand, $(1,1)$ is an
equilibrium payoff, as soon as $c<\frac32$, provided the players
are patient enough. Indeed, consider the strategy of the receiver
in which he matches the announcement of the sender, as long as
$|a_{n+1} - a_n| = 1$ modulo 5, and switches
forever to the babbling equilibrium (\textit{e.g.}, playing always $b=4$)
if $|a_{n+1} - a_n| \neq 1$ modulo 5 for some
stage $n$. Provided $c$ is not too large, the best response of the
sender is to report the true state. If instead, say, the sender
chooses to report $t=1$ when in fact $s=0$ in a given stage, he
gains $c-1$, but then in the next period, with probability
$\frac{1}{2}$ the new state will be $s=4$, and then he will either
report $t \in \{0,2\}$ and receive 0, or report $t \in \{1,3,4\}$
and be punished with the babbling equilibrium payoff $\frac{1}{5}$
forever. Provided the players are patient enough, such a deviation
is not profitable.

In this equilibrium, the receiver checks that the one-step
transitions between \emph{successive} announcements are consistent
with the transitions of the Markov chain. As it turns out, under
\textbf{Assumption A}, such a sophisticated statistical analysis
of the announcements is not more powerful than a statistical
analysis which is based only on the empirical frequencies of the
different announcements. \myendexample

\bigskip

Theorems \ref{theorem1} and \ref{theorem2} hold as soon as the sender knows the current state. As will be clear from the proof, they continue
to hold if the sender knows \emph{more}.
In particular, they hold in the extreme case where  the sender learns the entire sequence of realized states in stage 1, or in any intermediate setup.

\bigskip

Note that we interpret the case $\delta\to 1$ as players being very patient. It is  not possible here to interpret it as a situation in which players would interact more and frequently. Indeed, a proper analysis of this latter case would take into account the impact on transitions: when players interact more frequently, states become more persistent between successive interactions.

\subsection{An illustration}

We here analyze a simple, specific example to show how to pin down the set of equilibrium payoffs using our results. We let the set of states be $S=\{s_0,s_1,s_2\}$. Between any two stages, the state changes with probability one, and each of the two possible states is equally likely. Thus, $p(t\mid s)=\frac12$ for every $s\neq t\in S$.  Note that \textbf{Assumption A} on the transition function does hold, and that the invariant measure assigns probability $\frac13$ to each state.

 The receiver has three actions, denoted $L, M$ and $R$, and the payoffs in the different states are given by the matrix
\[\left(\begin{array}{ccc}
1,1 & 0,0 & 0,0 \\
0,0 & 1,0 & 0,1 \\
0,0 & 0,1 & 1,0 
\end{array}
\right)
\]
where each row corresponds to a state, and each column to an action. For instance, the first row specifies the payoffs in state $s_0$, as a function of the action of the receiver. 

All extreme points of the feasible set are obtained by having the sender report truthfully the state, and  
the receiver then play a pure, state-dependent, action. Thus, all extreme points are obtained by picking one entry in each row, and averaging. For instance,  picking $L$ (resp. $M$, $R$) in row $s_0$ (resp. $s_1$, $s_2$) and averaging over states leads to a payoff of $(1,\frac13)$. One checks that the feasible set is 
the convex hull of the five payoffs $(0,0)$, $(0,\frac23)$, $(\frac13,1)$, $(\frac23,0)$ and $(1,\frac13)$, see Figure \arabic{figurecounter} below.

\begin{center}
\begin{tikzpicture}[scale = 3]

\draw[thin,->] (0,0) -- (1.5,0); 
\draw (1.5,0) node[right] {$\gamma^1$}; 
\draw [thin,->] (0,0) -- (0,1.5); 
\draw (0,1.5) node[above] {$\gamma^2$};

\fill[color=gray!20] (0,0) -- (0,2/3)  -- (1/3,1)  -- (1,1/3)  -- (2/3,0) --cycle;

\draw (0,2/3)node[left]{(0,${2\over 3}$)};
\draw  (1/3,1)node[above]{(${1\over 3}$,1)};
\draw  (1,1/3)node[right]{(1,${1\over 3}$)};
\draw (2/3,0)node[below]{(${2\over 3}$,0)};

\draw (1/3,1) node[right]{The feasible set};
\draw (3/4,7/12) node[right]{The equilibrium set};

\draw(2/3,-1/4) node[below]{Figure \arabic{figurecounter}};

\draw [fill=gray!80](1/3,1/3) -- (2/3,2/3) -- (1,1/3) -- (1/3,1/3);
\draw (1/3,1/3) -- (1,1/3);
\end{tikzpicture}
\end{center}
Without any information on the state, all three actions of the receiver yield $\frac13$, hence 
$v^2=\frac13$.

Let $\gamma=U(\mu_0,y)\in E(\calM)$ be a (limit) equilibrium payoff, and denote by $\mu$ the copula obtained when the sender exchanges the two states $s_1$ and
$s_2$ when reporting. Thus, $\mu(s_1,s_2)=\mu(s_2,s_1)=\mu(s_0,s_0)=\frac13$, and $\mu(s,t)=0$ otherwise. Since the payoffs of the two players are exchanged in the two states $s_1$ and $s_2$, one has 
$U^1(\mu,y)=U^2(\mu_0,y)$ and $U^1(\mu_0,y)=U^2(\mu,y)$. The incentive condition $U^1(\mu_0,y)\geq U^1(\mu,y)$ thus yields $\gamma^1\geq \gamma^2$.

Note finally that the sum of the players' payoffs cannot exceed 2 in state $s_0$, and 1 in states $s_1$ and $s_2$. Thus, $\gamma^1+\gamma^2\leq \frac13 (2+1+1)=\frac43$.

Hence, any equilibrium payoff $(\gamma^1,\gamma^2)$ lies in the shaded triangle  defined by the inequalities $\gamma^2\geq \frac13$, $\gamma^1+\gamma^2\leq \frac43$, $\gamma^1\geq \gamma^2$, see Figure \arabic{figurecounter} below. 

On the other hand, each of the extreme points of this triangle is an equilibrium payoff. Indeed, $(\frac13,\frac13)$ is the babbling equilibrium payoff, while $(\frac23,\frac23)=U(\mu_0,y_1)$ and $(1,\frac13)=U(\mu_0,y_2)$, where $y_1:A\to \Delta(B)$  plays $L$ when told $s_0$, and randomizes between $M$ and $R$ otherwise, while $y_2$ plays $L$, $M$ and $R$ in states $s_0$, $s_1$ and $s_2$ respectively. 

As a result, the set of equilibrium payoffs is equal to the shaded triangle in Figure \arabic{figurecounter}.

\addtocounter{figurecounter}{1}

\section{Proofs}
\label{section proof}

\subsection{Proof of Theorem \ref{theorem1}}\label{sec_prooftheorem1}

We here provide most details of the proof of Theorem \ref{theorem1}. Some technical details are in the Appendix.
Since $\widehat{E}(\calM)\neq \emptyset$, there is
$y_0:A\to \Delta(B)$ such that
 $U^2(\mu_0,y_0)>v^2$ and $U^1(\mu_0,y_0)>U^1(\mu,y_0)$ for every $\mu\in \calM$, $\mu\neq \mu_0$.

Let $\ep >0$ and $U(\mu_0,\bar{y})\in E(\calM)$ be arbitrary, and define $y:=\ep y_0 +(1-\ep )\bar y$. It is sufficient to prove that $U(\mu_0,y)$ is arbitrarily close to some sequential equilibrium payoff of the $\delta$-discounted game, provided $\delta$ is high enough.

\subsubsection{The strategies}
\label{section strategies}

Let some integer $N\in \dN$, and a discount factor $\delta$ be given. We here define a strategy profile  $(\sigma_*,\tau_*)$. 

According to $(\sigma_*,\tau_*)$, the play is divided into consecutive blocks of $N$ stages. At the beginning of each block, players discard past information, and re-start playing a $N$-stage profile $(\sigma_0,\tau_0)$, where $\tau_0$ is a \emph{pure} strategy. In case the receiver deviates from the pure strategy $\tau_0$, the players switch to babbling play forever. 

We now construct $\sigma_0$ and $\tau_0$, starting with $\tau_0$. Consider any block of $N$ stages. According to $\tau_0$, the receiver "listens" to the reported state $a_n$  in stage $n$ and plays $y(\cdot\mid a_n)$, as long no state  has been reported too often. As soon as this fails to be the case, the receiver substitutes to the actual report of the sender some fictitious report $\theta_n$, and plays according to $y(\cdot \mid \theta_n)$. 

To be formal, we pick a distribution  $m_N\in \Delta(S)$ which best approximates the invariant measure $m$, among all distributions $\tilde m\in \Delta(S) $ such that $N\tilde m(s)$ is an integer for all $s$. 

For $s\in S$ and $n\in \dN$,
we  denote by $\textbf{N}_n(s)=|\{k\leq n: a_k=s\}|$ the number of stages where the sender reported state $s$, and we set
\[q:=\min\{1\leq n\leq B: \textbf{N}_n(a_n)>Nm_N(a_n)\}\]
($\min\emptyset=+\infty$). Intuitively, each
state $s\in S$ is allotted a quota of announcements 
equal to $N\; m_N(s)$. The stage $q$ is the first
stage in which quotas are no longer met.  From stage $q$ until the end of the block, the receiver substitutes fictitious reports to actual ones.

Formally, we let $(\theta_n)$ ($n=1,\ldots, N$) be a sequence such that
\begin{description}
\item[F1.] $\theta_n=a_n$ for $n<q$;
\item[F2.] For each $s\in S$, the equality $|\{n\leq N: \theta_n=s\}|=Nm_N(s)$ always hold;
\item[F3.] Conditional on $(a_1,\ldots, a_q)$, the variables $(\theta_{q},\ldots, \theta_N)$ are deterministic.
\end{description}
We will refer to $\theta_n$ as the \emph{announcement} at
stage $n$. Condition \textbf{F1} means that the announcements
coincide with the sender's actual announcements prior to stage $q$; condition
\textbf{F2} ensures that the entire sequence of announcements
always satisfies the quotas; condition \textbf{F3} ensures in
particular that the  fictitious announcements are commonly known
between the two players. 

Thanks to the public randomizing device, the strategy $\tau_0$ may be rewritten as a pure strategy.
\footnote{Indeed, denote by $X_n\sim \calU([0,1])$ the output of the public device in stage $n$, 
and label the receiver's actions
from 1 to $|B|$. We let the strategy $\tau_0$ instruct the receiver to
choose the action $b\in B$ whenever $\displaystyle
\sum_{i=1}^{b-1}y(i \mid \theta_n)\leq X_n <\sum_{i=1}^{b}y(i \mid
\theta_n)$.   In effect, the device is performing publicly the desired randomization.}

\bigskip

The strategy  $\sigma_0$ is defined to be any pure best-reply strategy of the sender to $\tau_0$ in the $N$-stage
$\delta$-discounted game starting in stage 1. Note that the strategy $\sigma_0$ is also a best-reply  to $\tau_0$ on any of the consecutive blocks of $N$ stages, conditional on past play.\footnote{This observation relies on the fact that, in the first block, $\sigma_0$ is a best-reply to $\tau_0$, no matter what 
the distribution of the initial state $\textbf{s}_1$ is.} In particular, $\sigma_*$ is a best-reply to $\tau_*$.

\subsubsection{Equilibrium properties}

We here argue that, for appropriate choices of $N$ and of $\delta$, $\tau_*$ is a 
best-reply to $\sigma_*$, and $(\sigma_*,\tau_*)$ induces a payoff arbitrarily close to $U(\mu_0,y)$.\footnote{Off-equilibrium path beliefs and sequential rationality issues are discussed in the Appendix.}

\begin{proposition}\label{prop eq}
For every $\eta>0$, there exists $N_0\in \dN$ such that the following holds.  For every $N\geq \dN$, 
there is $\delta_0<1$ such that, for every $\delta \geq \delta_0$, the profile $(\sigma_*,\tau_*)$ is a sequential equilibrium and induces a payoff within $\eta$ of $U(\mu_0,y)$. 
\end{proposition}

The complete proof of Proposition \ref{prop eq} is in the Appendix. The crucial step consists in showing that the fact that $\sigma_0$ is a best-reply to $\tau_0$ \emph{implies} that, with high probability, $\sigma_0$ reports the true state in most stages. 
This is the content of Lemma \ref{lemmmain} below. In the statement of the lemma,
 $\mu_{\sigma_0,\tau_0}$ is the (expected, undiscounted) joint distribution of states and reports in 
 a block of $N$ stages. That is, for each $(s,a)\in S\times S$,
\[\mu_{\sigma_0,\tau_0}(s,a):=\E_{\sigma_0,\tau_0}\left[\frac{1}{N}\sum_{n=1}^N1_{\{s_n=s,\theta_n=a\}}\right]\]
is the expected frequency of the pair $(s,a)$ over $N$ stages (recall that the  distribution of the initial state, $s_1$, is the invariant measure). 

\begin{lemma}\label{lemmmain}
For every $\eta >0$, there is $N_0\in \dN$, such that the following holds. For every $N\geq N_0$, there is $\delta_0<1$, such that,
for every $\delta\geq \delta_0$, one has
\[\|\mu_{\sigma_0,\tau_0}-\mu_0\|<\eta.\]
\end{lemma}

We will provide insights into the proof of Lemma \ref{lemmmain} below. For the time being, we show how to 
deduce  Proposition \ref{prop eq} from  Lemma \ref{lemmmain}.  Observe first that, by  definition of $\mu_{\sigma_0,\tau_0}$, the expected \emph{average}%
\footnote{That is, when payoffs in the different stages are not
discounted.} payoff induced by $(\sigma_0,\tau_0)$ over a single
block is equal to $U(\mu_{\sigma_0,\tau_0},y)$. For fixed $N$, and since the profile
$(\sigma_*,\tau_*)$ consists in periodic repetitions of $(\sigma_0,\tau_0)$, the discounted payoff induced by
$(\sigma_0,\tau_0)$ therefore converges to $U(\mu_{\sigma_0,\tau_0},y)$  as $\delta \to 1$. In particular, it
is thus arbitrarily close to the target payoff $U(\mu_0,y)$.

We now argue that $\tau_*$ is   a best-reply to $\sigma_*$. Following any history, the continuation payoff of the receiver is equal to the sum of his payoffs until
the end of the current block and of the continuation payoff from
the next block on. The latter is equal to the discounted payoff induced by
$(\sigma_*,\tau_*)$, computed using the belief held by the receiver at the beginning of this block. For fixed $N$,
this continuation payoff thus converges to
$U^2(\mu_{\sigma_0,\tau_0},y)$ as $\delta \to 1$.\footnote{Uniformly over all histories.}

On the other hand, any deviation from $\tau_0$, say in stage $n$, triggers a babbling play, and the
receiver's continuation payoff therefore does not exceed
$\displaystyle (1-\delta)\sum_{k=n}^{\infty}\delta^{k-n}\max_{b\in
B}u^2(p_k,b)$, where $p_k$ is the belief that the receiver will
hold at stage $k\geq n$ on the current state $s_k$. Since the 
sequence of states forms an irreducible and aperiodic chain, $p_k$ converges to $m$. For fixed $N$, 
this continuation payoff therefore converges to $v^2$ as $\delta\to 1$ (again, uniformly over all histories).  

Since
$U^2(\mu_{\sigma_0,\tau_0},y)>v^2$, this proves the best-reply
property of $\tau_0$, provided first $N$, and then $\delta$,
are chosen large enough.

\bigskip

We now turn to  Lemma \ref{lemmmain}.  We denote by $\sigma_{truth}$ the strategy of the sender that announces truthfully the current  state, no matter what. The proof of Lemma \ref{lemmmain} combines several ideas. 

First, by a law of large numbers for Markov Chains, and  if $N$ is large enough, there is a high probability that the realized state frequencies will be consistent with the quotas in most stages. Thus, under $(\sigma_{truth},\tau_0)$, there is a high probability that the receiver  follows the announcement of the sender in most stages. That is, the distribution $\mu_{\sigma_{truth},\tau_0}$ is arbitrarily close to $\mu_0$. 

Next, for fixed $N$, and for every (periodic) strategy $\sigma$ (and viewing $\tau_0$ as a periodic strategy), the discounted payoff $\gamma_\delta (\sigma,\tau_0)$ converges to $U(\mu_{\sigma,\tau_0}, y)$ as $\delta $ converges to one.

Finally, the best-reply property of $\sigma_0$ implies that $\gamma^1_\delta (\sigma_0,\tau_0)\geq \gamma^1_\delta(\sigma_{truth},\tau_0)$. 

Combining these observations, the following formal statement holds. For every $\ep>0$, there exist $N_0$ and $\delta_0$ such that, for every $\delta \geq \delta_0$, the following sequence of inequalities holds: 
\begin{equation}\label{marre}U^1(\mu_{\sigma_0,\tau_0},y)\geq \gamma^1_\delta(\sigma_0,\tau_0)-\ep \geq  \gamma^1_\delta(\sigma_{truth},\tau_0)-\ep \geq U^1(\mu_{\sigma_{truth},\tau_0},y)-2\ep
\geq  U^1(\mu_0,y)-3\ep.\end{equation}

To conclude, we will rely on Lemma \ref{lemm2} below, which critically depends on the assumption that 
$U(\mu_0,m_0)\in \widehat{E}(\calM)$. 
Denote by $\calM_e$ the (finite) set of extreme points of $\calM$.
Recall that $\mu_0 \in \calM_e$.
Set
\[c_1:=\min_{\{\mu_e \in \calM_e, \mu_e \neq \mu_0\}}\left(U^1(\mu_0,y_0)-U^1(\mu_e,y_0)\right),
\mbox{ and }c_2:=\max_{\{\mu_e \in \calM_e, \mu_e \neq \mu_0\}}\|\mu_e-\mu_0\|_1, \] and note that both $c_1$ and $c_2$ are positive.

\begin{lemma}\label{lemm2}
For every $\mu\in \calM$, one has
\[ U^1(\mu_0,y)-U^1(\mu,y)\geq \frac{\ep c_1}{c_2}\| \mu-\mu_0\|_1.\]
\end{lemma}

Lemma \ref{lemm2} may be paraphrased as saying that any strategy that does approximately as well as the truth-telling strategy must 
be telling the truth in most stages, with high probability. 

The conclusion of Lemma \ref{lemmmain} follows from (\ref{marre}) combined with Lemma \ref{lemm2}. 

\subsection{Proof of Theorem \ref{theorem2}}

We here provide insights into the proof  of Theorem
\ref{theorem2}. We let $\delta<1$, and we fix a Nash equilibrium
$(\sigma,\tau)$ of the $\delta$-discounted game (with or without randomizing device). 
For clarity,  we sometimes use boldfaced letters to denote random variables.

For $s\in S$, we
define $y(\cdot \mid s)\in \Delta(B)$ as the expected
\emph{discounted} distribution of moves of the receiver in state
$s$. Formally, for $s\in S$, and $b\in B$, we set
\[
y(b\mid s)=
\frac{1}{m(s)}\E_{\sigma,\tau}\left[\sum_{n=1}^\infty(1-\delta)\delta^{n-1}1_{\{\textbf{s}_n=s,\textbf{b}_n=b\}}\right]
=\frac{\E_{\sigma,\tau}\left[\sum_{n=1}^\infty(1-\delta)\delta^{n-1}1_{\{\textbf{s}_n=s,\textbf{b}_n=b\}}\right]}{
\E_{\sigma,\tau}\left[\sum_{n=1}^\infty(1-\delta)\delta^{n-1}1_{\{\textbf{s}_n=s\}}\right]}.
\]

By construction, one has  $\gamma_\delta(\sigma,\tau)=U(\mu_0,y)$.
Indeed,
\begin{eqnarray*}
\gamma_\delta(\sigma,\tau)&=& (1-\delta)\sum_{n=1}^\infty \delta^{n-1}\E_{\sigma,\tau}[u(\textbf{s}_n,\textbf{b}_n)]=
(1-\delta)\sum_{n=1}^\infty \delta^{n-1}\sum_{s\in S,b\in B}\E_{\sigma,\tau}[1_{\{\textbf{s}_n=s,\textbf{b}_n=b\}}] u(s,b)\\
&=& \sum_{s\in S,b\in B}u(s,b)m(s) y(b\mid s)= U(\mu_0,y),
\end{eqnarray*}
as desired.

Since the distribution of $\textbf{s}_n$ is equal to $m$ for each
stage $n \in \dN$, one has $\gamma^2_\delta(\sigma,\tau)\geq v^2$,
and thus, $U^2(\mu_0,y)\geq v^2$, so that \textbf{C2} holds.

We thus need to prove that $U^1(\mu_0,y)\geq U^1(\mu,y)$ for each $\mu\in \calM$. The idea of the proof is rather straightforward, but the formal proof is fraught with many 
technical complications. Let $\mu\in \calM$ be given. We will construct a strategy $\sigma'$ of the sender such that $\gamma_\delta(\sigma',\tau)=U(\mu,y)$, so that the desired inequality will follow from the equilibrium property of $(\sigma,\tau)$. 

The strategy $\sigma'$ is designed as follows. Along the play, the sender will generate a sequence $(\mathbf{t}_n)$ of \emph{fictitious} states  that is statistically indistinguishable from the sequence $(\mathbf{s}_n)$, and such that the average distribution of the pair $(\mathbf{s}_n,\mathbf{t}_n)$ is given by $\mu$. Given such a sequence, in any stage $n$, the sender will substitute the fictitious state $\mathbf{t}_n$ to the realized state $\mathbf{s}_n$ in playing $\sigma$. Formally, following any history $(s_1,t_1,a_1,b_1,\ldots, s_n,t_n)$ consisting of realized and fictitious states, messages and actions up to stage $n$, the strategy $\sigma'$ plays the mixed move $\sigma(t_1,a_1,b_1\ldots, t_n)$ that would have been played by $\sigma$, had the realized states been $t_1,\ldots, t_n$. 

We now give some more details. Since the strategy $\tau$ may feature complex statistical tests on the successive announcements, the notion of being statistically indistinguishable has to be interpreted in a restrictive sense. 

We prove in the Appendix the following lemma.

\begin{lemma}\label{lemm4ter}
Assume Assumption A, and let $\mu\in \calM$ be given. There exists an  $S$-valued process%
\footnote{The process $(\mathbf{t}_n)_n$ is possibly defined on a probability space which
is an enlargement of the one on which $(\mathbf{s}_n)_n$ is defined.}
$(\textbf{t}_n)_n$,
such that:
\begin{description}
\item[P1]  Conditional on $\textbf{s}_n$, the
vector $(\textbf{t}_1,\ldots, \textbf{t}_n)$ is independent of the
future states $(\textbf{s}_{n+1},\textbf{s}_{n+2},\ldots)$.

 \item[P2]  The law of the sequence $(\textbf{t}_n)_n$ is the same
as the law of the sequence $(\textbf{s}_n)_n$.

\item[P3] The law
of the pair $(\textbf{s}_n,\textbf{t}_n)$ is $\mu$, for each stage
$n\in \dN$.
\item[P4] The conditional law of $\textbf{s}_n$, given
$\textbf{t}_1,\ldots, \textbf{t} _n$ is $\mu(\cdot \mid
\textbf{t}_n)$. 
\end{description}
\end{lemma}

According to \textbf{P1}, the state $\mathbf{t}_n$ can be computed/simulated using only the information available at stage $n$: 
past and current states, and past fictitious states. This is a feasibility requirement that ensures that $\sigma'$ is well-defined. 
Condition \textbf{P2} ensures that no statistical test can discriminate between the sequences $(\mathbf{s}_n)$ and $(\mathbf{t}_n)$. 
Condition \textbf{P3} provides the desired coupling between $\mathbf{s}_n$ and $\mathbf{t}_n$. 
\bigskip

We now proceed to show that the expected payoff induced by $(\sigma',\tau)$ is then equal to $U(\mu,y)$, as claimed.

Below we will denote by $s_k, t_k, b_k$ generic values of the
random variables $\textbf{s}_k,\textbf{t}_k$ and $\textbf{b}_k$,
respectively. For any given stage $n \in \dN$, the following
sequence of equalities holds:
\begin{eqnarray}
\nonumber
\E_{\sigma',\tau}[u(\textbf{s}_n,\textbf{b} _n)]&=& \sum_{s_n,b_n}\prob_{\sigma',\tau}(\textbf{s}_n=s_n,\textbf{b}_n=b_n)u(s_n,b_n)\\
&=& \sum_{s_n}\sum_{t_1,\ldots, t_n}\sum_{b_1,\ldots, b_n}\prob_{\sigma',\tau}(s_n,t_1,\ldots, t_n,b_1,\ldots, b_n)u(s_n,b_n)\\
\nonumber
&=& \sum_{s_n}\sum_{t_1,\ldots, t_n}\sum_{b_1,\ldots, b_n}\prob_{\sigma',\tau}(s_n\mid t_1,\ldots, t_n,b_1,\ldots, b_n)
\prob_{\sigma',\tau}(t_1,\ldots, t_n,b_1,\ldots, b_n)
u(s_n,b_n)\\
\label{equ7.1}
&=& \sum_{s_n}\sum_{t_1,\ldots, t_n}\sum_{b_1,\ldots, b_n}\prob(s_n\mid t_1,\ldots, t_n)\prob_{\sigma',\tau}(t_1,\ldots, t_n,b_1,\ldots, b_n)u(s_n,b_n)\\
\label{equ7.2}
&=& \sum_{s_n}\sum_{t_1,\ldots, t_n}\sum_{b_1,\ldots, b_n}\mu(s_n\mid t_n)\prob_{\sigma',\tau}(t_1,\ldots, t_n,b_1,\ldots, b_n)u(s_n,b_n)\\
&=& \sum_{s_n,t_n,b_n}\mu(s_n\mid t_n)\prob_{\sigma',\tau}(t_n,b_n)u(s_n,b_n),\\
\label{equ5.1}
\end{eqnarray}
where (\ref{equ7.1}) holds because the variables $(\mathbf{b}_1,\ldots, \mathbf{b}_n)$ are conditionnally  independent of $s_n$
given $(\mathbf{t}_1,\ldots, \mathbf{t}_n)$, 
and (\ref{equ7.2}) holds by \textbf{P4}.

Using \textbf{P2}, and by the definition of $\sigma'$, the $\delta$-discounted sum of $\prob_{\sigma',\tau}(\mathbf{t}_n=t_n,\mathbf{b_n}=b_n)$
is equal to $\prob_{\sigma,\tau}(\mathbf{s}_n=t_n,\mathbf{b_n}=b_n)$, which is equal to $\mu(t_n)\times y(b_n\mid t_n)$.
By (\ref{equ5.1}) we now obtain
\[\gamma_\delta(\sigma',\tau)= \sum_{s,t,b}\mu(s\mid t)\mu(t)y(b\mid t)u(s,b)=U(\mu,y).\]

\section{Further results and comments}\label{sec_furtherf}

\subsection{On the condition $\widehat{E}(\calM)\neq \emptyset$}\label{sec_generic}

In the light of  existing results for repeated games, it is not
surprising that some non-empty interiority type of assumption is
needed (see Mailath and Samuelson (2006) for a survey).

As the next example illustrates, the conclusion of Theorem \ref{theorem1} fails to hold if $\widehat{E}(\calM)=\emptyset$.
\begin{example}
\label{example6} Let there be two states and two actions for the
receiver. The payoffs in the two states are given by the tables in
Figure \arabic{figurecounter}. We assume that the successive
states are independent and that the two states are equally likely.

\centerline{
\begin{picture}(90,60)(-10,-20)
\put( 20,27){$l$}
\put( 60,27){$r$}
\put( 0,0){\numbercellong{$0.5,1$}}
\put(40,0){\numbercellong{$1,1$}}
\put(20,-15){State $L$}
\end{picture}
\ \ \ \ \ \ \ \ \ \ \ \ \ \ \ \ \ \ \ \
\begin{picture}(90,60)(-10,-20)
\put( 20,27){$l$}
\put( 60,27){$r$}
\put( 0,0){\numbercellong{$0,0$}}
\put(40,0){\numbercellong{$1,1$}}
\put(20,-15){State $R$}
\end{picture}
} \centerline{Figure \arabic{figurecounter}: The game in Example
\ref{example6}.}
\end{example}
\addtocounter{figurecounter}{1}

The strategy which plays $r$ irrespective of the announcement is
weakly dominant in the one-shot game, and thus, $v^2=1$. Consider
now the stationary strategy $y$ defined by $y(l\mid L)=y(r\mid
R)=1$. The payoff vector $U(\mu_0,y)=(\frac{3}{4},1)$ is in
$E(\calM)$. However, we claim that $(1,1)$ is the unique
equilibrium payoff, irrespective of $\delta$. Here is why.
Consider any equilibrium $(\sigma,\tau)$. Plainly, the equilibrium
payoff of the receiver is equal to 1. In particular, with
probability 1 the receiver plays $r$ whenever the current state is
$R$. This implies that in every stage, and for \textit{a.e.} past history,
there is one (possibly history-dependent) message following which
the receiver plays $r$, and which is assigned positive probability
by $\sigma$. But then, the sender gets a payoff 1 by
assigning probability 1 to this specific message in every stage.
\myendexample

\bigskip

As we stressed, the statement of Theorem \ref{theorem1} is  unsatisfactory in one important respect:
while non-empty interior requirements in existing Folk Theorems are generically satisfied,
the condition $\widehat{E}(\calM)\neq \emptyset$ does \emph{not} hold generically, as the next example shows.

\begin{example}
\label{example5} Consider the game depicted in Figure
\arabic{figurecounter}, where there are two states, and the
receiver has two actions.

\centerline{
\begin{picture}(90,60)(-10,-20)
\put( 20,27){$l$}
\put( 60,27){$r$}
\put( 0,0){\numbercellong{$1,1$}}
\put(40,0){\numbercellong{$0,0$}}
\put(20,-15){State $L$}
\end{picture}
\ \ \ \ \ \ \ \ \ \ \ \ \ \ \ \ \ \ \ \
\begin{picture}(90,60)(-10,-20)
\put( 20,27){$l$}
\put( 60,27){$r$}
\put( 0,0){\numbercellong{$1,1$}}
\put(40,0){\numbercellong{$0,0$}}
\put(20,-15){State $R$}
\end{picture}
} \centerline{Figure \arabic{figurecounter}: The game in Example
\ref{example5}.}
\end{example}
\addtocounter{figurecounter}{1}

Here, $v^2=1$, and the stationary strategy $y_*$ which plays $l$
irrespective of the announcement is the only stationary strategy
that satisfies \textbf{C2}. Hence, $E(\calM)$ contains a single
payoff vector, $(1,1)$, and $\widehat E(\calM)$ is empty. When
payoffs are slightly perturbed, the strategy $y_*$ remains the
only strategy satisfying \textbf{C2}, therefore $\widehat
E(\calM)=\emptyset$ for any such perturbation. \myendexample

\bigskip

Example \ref{example5} suggests that if all  strategies $y$ for which $U(\mu_0,y)$ is in $E(\calM)$ are constant strategies,
then the set $\widehat E(\calM)$ is empty, even when payoffs in the game are slightly perturbed. We build on this intuition, and introduce a
new condition.

\begin{description}
\item[Condition B.] There is a \emph{non-constant} map $y:S\to \Delta(B)$ such that $U(\mu_0,y)\in E(\calM)$.
\end{description}

If \textbf{condition B} is not met, then all equilibrium payoffs are babbling.

In Theorem \ref{theorem3} below, we fix the transition function of
the Markov chain $p$, and identify a game to a point in the space
$\dR^{2\times S\times B}$ of payoff functions.

\begin{theorem}\label{theorem3}
Let a game $G$ be given.

If condition \textbf{B} holds for $G$, then any neighborhood of
$G$ contains a game $G'$ with $\widehat{E}_{G'}(\calM)\neq
\emptyset$.

If condition \textbf{B} does not hold for $G$, there is a
neighborhood $\calN$ of $G$ such that, for every game in $\calN$,
condition \textbf{B} does not hold.
\end{theorem}

Theorem \ref{theorem3} allows us to complete the picture provided
by Theorems \ref{theorem1} and \ref{theorem2},
provided the underlying Markov chain satisfies  \textbf{Assumption A}. Indeed, let $G$ be
a game. If Condition B holds for  the game $G$, Theorems \ref{theorem1} and
\ref{theorem2} provide a characterization of the limit set of
equilibrium payoffs for games arbitrarily close to $G$. If
condition B does not hold for the game $G$, then all games close enough to
$G$ have only babbling equilibrium payoffs.

\subsection{On the role of the randomizing device}\label{sec_device}

The randomizing device is not needed in the proof of Theorem
\ref{theorem2} to implement payoffs $U(\mu_0,y)$, whenever
$y(\cdot\mid s)$ is a pure strategy: it assigns probability 1 to
some action $b(s)$, for each $s\in S$. However, as soon as
$y(\cdot\mid s)$ is a truly mixed distribution for some state $s$,
it may be impossible to dispense with the randomizing device, as
we now argue by means of an example.

Let there be two states, $L$ and $R$. The successive states are
drawn independently in every period, and each of the two states is
equally likely. The receiver has three actions, denoted
$B=\{l,m,r\}$. The payoffs are given in Figure
\arabic{figurecounter}.

\centerline{
\begin{picture}(90,60)(-10,-20)
\put( 20,27){$l$}
\put( 60,27){$m$}
\put(100,27){$r$}
\put( 0,0){\numbercellong{$3,0$}}
\put(40,0){\numbercellong{$0,4$}}
\put(80,0){\numbercellong{$2,1$}}
\put(50,-15){State $L$}
\end{picture}
\ \ \ \ \ \ \ \ \ \ \ \ \ \ \ \ \ \ \ \
\begin{picture}(90,60)(-10,-20)
\put( 20,27){$l$}
\put( 60,27){$m$}
\put(100,27){$r$}
\put( 0,0){\numbercellong{$1,-5$}}
\put(40,0){\numbercellong{$4,-4$}}
\put(80,0){\numbercellong{$2,1$}}
\put(50,-15){State $R$}
\end{picture}
} \centerline{Figure \arabic{figurecounter}: The payoffs of the
players.}

\addtocounter{figurecounter}{1}

Plainly, $v^2=1$. Define $y_*$ to be the stationary strategy such
that $y_*(\cdot\mid R)$ assigns probability 1 to $r$, and
$y_*(\cdot\mid L)$ assigns probabilities $\frac23$ and $\frac13$
to $l$ and $m$, respectively. Then $U(\mu_0,y)=(2,\frac76)$, and
one can verify that $U(\mu_0,y_*)\in E(\calM)$ while
$\widehat{E}(\calM)\neq\emptyset$. Thus, using Theorem
\ref{theorem1}, the vector $(2,\frac76)$ can be approximated by
sequential equilibrium payoffs, when players are sufficiently
patient, provided a randomizing device is available.

We now assume that such a device is not available. Since
successive states are independent, the dynamic game can be viewed
as a infinite repetition of the one-shot information transmission
game. With this interpretation, an \emph{action} of the sender in
the one-shot game is a map $x:S\to A$, while an action of the
receiver is a map $y:A\to B$. Given an  action profile $(x,y)$,
payoffs are random, and take the value $u(s,y(x(s)))$ with
probability $m(s)$, for $s \in S$. Players then receive the public
signal $(x(s),y(x(s)))$.

We will rely on Fudenberg, Levine and Maskin's (1994)
characterization of the limit set of perfect public equilibrium
(PPE) payoffs in repeated games with public signals. Some care is
needed, as there are two dimensions according to which our
repeated game does not fit into their setup. First, they assume
that a player's payoff depends deterministically on his own action
and on the public signal, while payoffs here depend randomly on
the entire action profile $(x,y)$. Second, their result is a
characterization of public equilibrium payoffs, while we focus on
sequential equilibrium payoffs.

We briefly argue that their result nevertheless applies to our setting. On the
one hand, their result is still valid for games where payoffs
depend on the entire action profile.\footnote{This can be seen
from their proof or, alternatively, deduced from H\"orner et al.
(2009).} Next, it can be verified that the auxiliary game in which
stage payoffs are defined to be the \emph{expected} stage
payoffs in our game (given the action profile) has the same set of
PPE payoffs. Thus, their result provides a characterization of the
limit set of PPE payoffs for our game. On the other hand, let
$(\sigma,\tau)$ be a sequential equilibrium of our game, and
define a public strategy profile $(\bar \sigma,\bar \tau)$ as
follows. Let any public history $\bar h$ be given. At $\bar h$, we
let $\bar \sigma$ play the expectation of the mixed move played by
$\sigma$, where the expectation is computed w.r.t. the belief held
by the receiver at the information set which contains $\bar h$. We
define $\bar \tau(h) $ by exchanging the roles of the two players.
It can be verified that $(\bar\sigma,\bar \tau)$ is a public
perfect equilibrium of the repeated game.

\bigskip


Fudenberg et al. (1994) showed that $\gamma \in \dR^2$ is a limit PPE
payoff if and only if  for all $\lambda \in \dR^2$ we have $\lambda\cdot \gamma
\leq k(\lambda)$, where  $k(\lambda)$ is the solution to   a certain optimization problem
$\calP(\lambda)$.\footnote{Their result requires that a certain set have a non-empty interior, a condition that can be checked to be met here.}

We set $\gamma=(2,\frac76)$, and we will show that it is not a PPE Payoff using the condition of
Fudenberg et al. (1994) with  $\lambda_*=(0,1)$.  We now  recall Fudenberg et al. (1994)
definition of $k(\lambda_*)$, and we will show that $\lambda_*\cdot \gamma
> k(\lambda_*)$, implying that $\gamma$  is not  a limit PPE payoff.

We denote by $Z = A \times B$ the
set of public signals in our game.
The quantity $k(\lambda_*)$ is defined as  the value of the optimization problem
$\calP$:
\[\sup V^2,\]
where the supremum is taken over all $(V^1,V^2)\in \dR^2$, and all
$\phi: Z\to \dR^2$, such that
\begin{itemize}
\item $\phi^2(z)\leq 0$ for every $z\in Z$; \item $(V^1,V^2)$ is a
Nash equilibrium payoff of the one-shot game, with payoff function
defined by:
\begin{equation}\label{oneshotgame}\sum_{s\in S}m(s)\left(u(s,y(x(s)))+\phi(x(s),y(x(s)))\right),\end{equation}
for each action pair $(x,y)$.
\end{itemize}

Let $\phi:Z\to \dR^2$ be any map such that $\phi^2(z)\leq 0$ for
each $z\in Z$, and let $(\alpha,\beta)$ be any (possibly mixed)
equilibrium of the one-shot game (\ref{oneshotgame}), with payoff
$(V^1,V^2)$. We will prove that $V^2<\frac76$. We argue by
contradiction, and assume that $V^2\geq \frac76$. We distinguish
between two cases.

Assume first that $\alpha: S\to \Delta(A)$ is pooling: the
distribution of messages is the same in both states. Then, since
$\phi^2(z)\leq 0$, the expected payoff of the receiver is not
higher than
\[\max_{b\in B}\sum_{s\in S}m(s)u^2(s,b)=1.\]
Thus, $V^2\leq 1 < \frac{7}{6}$, which is the desired contradiction.

Assume next that $\alpha$ is not pooling. Up to a relabelling of
the messages, we may then assume w.l.o.g. that the sender always
tells the truth with positive probability. That is, $\alpha(s\mid
s)>0$, for each $s\in S$. We denote by $\tilde \beta(\cdot\mid s)$
the conditional distribution of the receiver's move under
$(\alpha,\beta)$, conditional on the state being $s\in S$.
Denoting by $s\neq t$ the two states, the equilibrium property for
the sender in the game (\ref{oneshotgame}) then implies that
\[\sum_{b\in B}\beta(b\mid s)\left(u^1(s,b)+\phi^1(s,b)\right)\geq \sum_{b\in B}\beta(b\mid t)\left(u^1(s,b)+\phi^1(t,b)\right),\]
with equality if $\alpha(\cdot\mid s)$ assigns positive
probability to both messages, and
\[\sum_{b\in B}\beta(b\mid t)\left(u^1(t,b)+\phi^1(t,b)\right)\geq \sum_{b\in B}\beta(b\mid s)\left(u^1(t,b)+\phi^1(s,b)\right).\]
Using the two inequalities, one can verify that
\[u^1(s,\tilde\beta(\cdot\mid s))+u^1(t,\tilde \beta(\cdot\mid t))\geq u^1(s,\tilde\beta(\cdot\mid t))+ u^1(t,\tilde\beta(\cdot\mid s)).\]
By Lemma \ref{lemmbasic}, condition \textbf{C2} therefore holds for the stationary strategy $\tilde\beta$.

On the other hand, since $\phi^2(z)\leq 0$ for each $z$, the
expected payoff $V^2$ to the receiver does not exceed
$U^2(\mu_0,\tilde\beta)$. Hence,
$U^2(\mu_0,\tilde\beta)\geq\frac76$. This readily implies that
$U(\mu_0,\tilde\beta)\in E(\calM)$.

Next, one can verify that the highest payoff
$U^2(\mu_0,\tilde\beta)$ to the receiver, over the whole set
$U(\mu_0,\tilde\beta)\in E(\calM)$, is equal to $\frac76$. In
addition, the unique strategy $\tilde\beta$ that achieves such a
payoff is the strategy $y_*$. Since the  supports of $y_*(\cdot \mid L)$ and
$y_*(\cdot\mid R)$ are distinct, it must therefore be that
$\alpha$ is truth-telling: $\alpha(s\mid s)=1$ for each $s$.
Therefore, $\beta$ is equal to $ y_*$.

Since $V^2\geq \frac76$ and $V^2\leq U^2(\mu_0,y_*)$, one also has
$V^2=U^2(\mu_0,y_*)$. In particular, the expectation of
$\phi^2(z)$ under the equilibrium profile $(\alpha,\beta)$ must be
equal to zero. Since $\phi^2(z)\leq 0$ for each $z$, this implies
that $\phi^2(z)=0$, for each public signal $z$ that receives
positive probability under $(\alpha,\beta)$.

Using this, we finally claim that the equilibrium condition for
the receiver in the game (\ref{oneshotgame}) is violated. Indeed,
when told $L$, the strategy $\beta=y_*$ assigns positive
probability to both $l$ and $m$. Hence,
$\phi^2(L,l)=\phi^2(L,m)=0$ by the previous paragraph. On the
other hand however, $u^2(L,m)>u^2(L,l)$, hence the receiver is not
indifferent between both actions. This is the desired
contradiction.

\subsection{Imperfect monitoring}

Let us assume here that successive states are independent. Results
continue to hold if the receiver only observes a noisy,
\emph{public} version of the sender's message (provided the
definition of $U(\mu,y)$ is modified in an appropriate way). They
still hold if the receiver observes a noisy, \emph{public} signal
of the current state, provided the individual rationality level
$v^2$ is modified in the proper way. They also hold, without
changes, if the sender only observes a noisy, public signal of the
receiver's action. What happens in any of these variants when
signals are private is beyond the scope of the paper.

\bigskip

We briefly conclude this section by discussing the case where the sender fails to receive any information relative to the receiver's choices.
In spite of this feature, the game does not reduce to a sequence of successive, independent, one-shot games,
because of the ability of the receiver to monitor the sender. In particular,
it is easy to construct examples with equilibrium payoffs that lie outside of the convex hull of the set of equilibrium payoffs in the one-shot game.

We refer to the game where the sender does not observe the actions
of the receiver as to the {\it blind} game. Denote by
$NE^b_\delta$ the set of all Nash  equilibrium payoffs  of the
blind game.  We prove that the value of monitoring is positive, in
the sense that allowing the sender to monitor the receiver has a
non-ambiguous effect on the equilibrium set.

\begin{proposition}\label{prop subset}
The set $NE^b_\delta$ is a subset of $NE_\delta$.
\end{proposition}

\begin{proof}
Let $(\sigma,\tau)$ be a Nash equilibrium of the blind game.
Define $\tau'$ to be the following strategy that depends only on
the sender's announcements, and not on the receiver's past
actions: after a sequence $(a_1,\ldots, a_n)$ of announcements,
$\tau'$ plays any action $b \in B$ with the probability that the
$n$-th action of the receiver according to $\tau$ is $b$,
conditional on the sender's announcements being $(a_1,\ldots,
a_n)$:
\[ \tau'(a_1,\ldots, a_n)[b] = \E\left[\tau(a_1,b_1,a_2,b_2,\ldots, b_{n-1},a_n)[b] \mid a_1,a_2,\ldots,a_n\right]. \]
 In words,
$\tau'$ gets rid of the possible correlation between successive actions of the receiver, that may exist in the strategy $\tau$.

We claim that the strategy profile $(\sigma,\tau')$ is a Nash equilibrium of the blind game.
Indeed, $\tau'$ is a best-reply to $\sigma$ because it induces the same payoff as $\tau$.
$\sigma$ is a best-reply to
$\tau'$ because any strategy of the sender in the \emph{blind} game induces the same expected payoff against $\tau$ or $\tau'$.

We next claim that the strategy profile $(\sigma,\tau')$ is a Nash equilibrium of the non-blind game.
Indeed, because under $\sigma$, the sender does not condition his play on past actions of the receiver,
and because $\tau'$ is a best response to $\sigma$ in the blind game,
it follows that $\tau'$ is a best response to $\sigma$ in the non-blind game as well.
Because the receiver's actions are conditionally independent, given the sender's announcements,
any profitable deviation against $\tau'$ in the non-blind game is also profitable in the blind game.
\end{proof}

The inclusion is strict in general, as Example \ref{example-strict-inclusion} below shows.

\begin{example}\label{example-strict-inclusion}
There are two states $S = \{L,R\}$, and three actions for the
receiver, $B=\{l,m,r\}$. The payoffs in the two states are given
in Figure \arabic{figurecounter}.

\centerline{
\begin{picture}(90,60)(-10,-20)
\put( 20,27){$l$}
\put( 60,27){$m$}
\put(100,27){$r$}
\put( 0,0){\numbercellong{$2,2$}}
\put(40,0){\numbercellong{$0,0$}}
\put(80,0){\numbercellong{$0,0$}}
\put(50,-15){State $L$}
\end{picture}
\ \ \ \ \ \ \ \ \ \ \ \ \ \ \ \ \ \ \ \
\begin{picture}(90,60)(-10,-20)
\put( 20,27){$l$}
\put( 60,27){$m$}
\put(100,27){$r$}
\put( 0,0){\numbercellong{$0,0$}}
\put(40,0){\numbercellong{$2,2$}}
\put(80,0){\numbercellong{$0,3$}}
\put(50,-15){State $R$}
\end{picture}
} \centerline{Figure \arabic{figurecounter}: The game in Example
\ref{example-strict-inclusion}.}\end{example}
\addtocounter{figurecounter}{1}

We claim that $(2,2)$ is an equilibrium payoff when the sender
observes the actions of the receiver, but it is no longer an
equilibrium payoff when the sender does not observe the receiver's
actions.

Note first that $v^2=\frac{3}{2}$, and that $\widehat E(\calM)
\neq \emptyset$. By Theorem \ref{theorem1}, $(2,2) \in E(\calM)$,
so that $(2,2) \in \lim \inf_{\delta \to 1} SE_\delta \subseteq \lim \inf_{\delta \to 1} NE_\delta$.

We now argue that $(2,2)$ is bounded away from the set
$NE_\delta^b$. Indeed, assume to the contrary that there is some
equilibrium profile $(\sigma,\tau)$ of the blind game with a
payoff close to $(2,2)$. In particular, with a probability close
to one, there is a positive fraction of the stages in which the
current state is $R$ and the receiver plays $m$. Consider the
strategy $\tau'$ which plays as $\tau$, except that $\tau'$ plays
$r$ whenever $\tau$ would play $m$. Because the sender does not
observe the receiver's actions, he cannot tell whether the
receiver uses $\tau$ or $\tau'$, and therefore $\tau'$ is a
profitable deviation of the receiver: it yields the receiver
payoff close to $2\frac12$. \myendexample

\subsection{Relation to the one-shot game}

The characterization implies that every equilibrium payoff of the one-shot game remains an equilibrium payoff in the dynamic game, provided players are patient enough.
This property is not obvious \textit{a priori}, since the game is not a repeated game. In particular, it would typically fail to hold if the state
were constant throughout the play.

Let $(\sigma,\tau)$ be an equilibrium of the one-shot game. Let
$y:A\to \Delta(B)$ be the stationary strategy defined as
\[ y(b\mid a) = \sum_{s \in S} \sigma(a \mid s)\tau(a)[b]. \]
Note that the expected payoff under $(\sigma,\tau)$ is
$U(\mu_0,y)$. We claim that $U(\mu_0,y) \in E(\calM)$, so that by
Theorem \ref{theorem1} it is a sequential equilibrium payoff in
the repeated game. Indeed, because the receiver can guarantee
$v^2$ in the one-shot game, condition \textbf{C2} holds. Because
$\sigma$ is a best reply to $\tau$ in the one-shot game, the
inequality in \textbf{C1} holds for every $\mu$, and in particular
for every $\mu \in \calM$.

This result has the implication that the lowest equilibrium payoff of the sender in the repeated game
cannot be higher than his lowest equilibrium payoff in the one shot game. As the example in Section \ref{example2.1} shows, it can in fact be   strictly lower.

On the other hand, the lowest equilibrium payoff of the receiver
in both the one-shot game and the repeated game is equal to his
babbling equilibrium payoff $v^2$.

\renewcommand{\baselinestretch}{1.0}

\begin{appendix}

\begin{center}
\Large{\textbf{Appendix}}
\end{center}


\section{Proof of Lemma \ref{lemmbasic}}

To prove Lemma \ref{lemmbasic} we need the following description
of $\calM$, which is of independent interest.

A \emph{permutation matrix} is a (square) matrix with entries in $\{0,1\}$, such that each row and each column contains
exactly one entry equal to 1. We denote by $\Phi$ the set of $S \times S$ permutation matrices,
and by $I$ the matrix that corresponds to the identity permutation.
\begin{lemma} \label{lemma111}
The set $\calM(m)$ is equal to
$${\cal M}(m)= \left(\mu_0 -I + \mbox{co}\; \Phi \right) \cap \dR_+^{S\times S}.$$
\end{lemma}

\begin{proof}
The inclusion $\supseteq$ is clear. We prove the reverse inclusion.
Take $\mu$ in ${\cal M}(m)$, and define the matrix $J:=\mu+ I-\mu_0$ in $\dR^{S\times S}$.
$J$ is a bistochastic matrix, hence it is a convex combination of permutation matrices. Since $\mu=J-I+\mu_0$, the result follows.
\end{proof}

\bigskip

\noindent
\begin{proof}[Proof of Lemma \ref{lemmbasic}]
We only prove that \textbf{C1} is equivalent to \textbf{C'1}. For
every permutation $\phi$ over $S$ denote by $\mu^\phi \in S \times
S$ the matrix where the entry $(s,t)$ is equal to 1 if
$t=\phi(s)$, and is 0 otherwise. Note that $U^1(I,y) = \sum_{s \in
S}u^1(s,y(\cdot\mid s))$, and $U^1(\mu^\phi,y) = \sum_{s \in
S}u^1(s,y(\cdot\mid \phi(s)))$.

Assume first that \textbf{C'1} holds, and let $\mu \in \calM(m)$.
By Lemma \ref{lemma111}, $\mu$ can be written $\mu=\mu_0 -I+ \sum_{\phi} \alpha_\phi P^\phi$,
where the $\alpha_{\phi}$ are non negative real numbers that  sum to one.
Because $U^1$ is linear in $\mu$,
\[ U^1(\mu,y) = U^1(\mu_0,y) - U^1(I,y) + \sum_{\phi} \alpha_\phi U^1(\mu^\phi,y). \]
By \textbf{C'1}, $U^1(I,y) \geq U^1(\mu^\phi,y)$ for every
permutation $\phi$, and therefore $U^1(I,y) \geq \sum_{\phi}
\alpha_\phi U^1(\mu^\phi,y)$. It follows that $U^1(\mu_0,y) \geq
U^1(\mu,y)$. Because this inequality holds for every $\mu \in
\calM(\mu)$, \textbf{C1} holds.

Assume now that \textbf{C1} holds. Fix a permutation $\phi$, and
define $\mu_\ep=\mu_0- \varepsilon I + \varepsilon \mu^\phi$,
where $\ep > 0$. Because $m$ has full support, one has $\mu_\ep
\in \calM(m)$ provided $\ep$ is sufficiently small. Now, by
\textbf{C1}, for each such $\ep$,
\[ U^1(\mu_0,y) \geq U^1(\mu_\ep,y) = U^1(\mu_0,y) - \ep U^1(I,y) + \ep U^1(\mu^\phi,y). \]
It follows that $U^1(\mu^\phi,y) \leq \ep U^1(I,y)$.
As this inequality holds for every permutation $\phi$, \textbf{C'1} holds.
\end{proof}

\section{Complements to the proof of Theorem \ref{theorem1}}

The proof of Theorem \ref{theorem1} given in the text is almost complete. For completeness, we provide below the  proofs of  Proposition \ref{prop eq} and of Lemma \ref{lemm2}, which are missing. 

We start by  addressing the issue of designing a system of beliefs for the receiver that is consistent with $\sigma_*$, and that satisfies an additional  property.
Since the game involves randomizing devices with uncountably many outcomes, the standard definition of consistency does not apply.
We denote by $\lambda\in {\Delta}(S)$ a distribution with full support and, for $\eta<0$, we denote by $\sigma_\eta$ the strategy that,
following any history $h_n$, plays $\eta \lambda +(1-\eta)\sigma_*(h_n)$.

One can check that, for $\eta>0$, the beliefs of the receiver are uniquely defined by Bayes rule, and have a limit when  $\eta=0$.%
\footnote{And  the convergence is uniform w.r.t. the receiver's information set.}
Note that, following any history that is inconsistent with $\tau_0$,
the belief of the receiver in stage $n$ is independent of $\textbf{t}_n$.%
\footnote{That is, should the sender fail to play the babbling announcement $\bar a$, the receiver sill interprets the sender's announcements as babbling.}

We denote by  $\tau_*$ a strategy that coincides with $\tau_0$ as
long as the sender does not deviate, and that plays in each later
stage $n$ an action that (i) maximizes the current expected payoff
of the receiver, given the belief held by the receiver in stage
$n$, and (ii) does not depend on the announcements made by the
sender since the deviation took place.

By construction, the strategy $\sigma_*$ is sequentially rational at each information set of the sender,
while the strategy $\tau_*$ is sequentially rational at each information set of the receiver that is inconsistent with $\tau_0$.

\subsection{Proof of Proposition \ref{prop eq}}

Assume w.l.o.g. that all payoffs belong to the interval $[0,1]$. Define by $\sigma_{truth}$ the strategy
of the sender that announces truthfully the current state $\textbf{s}_n$ in each stage $n\leq N$, and by $\tau_{truth}$ the strategy of the receiver that
plays $y(\textbf{t}_n)$ in each stage $n\leq N$. Thus,  $\tau_{truth}$ coincides with $\tau_0$ until stage $q$.

Let $\eta$ be given, and set $\xi = \frac{\eta}{|S|+2}$. For every
state $s \in S$ and every $n \in \dN$, denote by $F_n(s)$ the
empirical frequency of visits to $s$ up to (and including) stage
$n$. Since the Markov chain is aperiodic, by the ergodic theorem
there is $N_0$ such that with probability at least $1-\xi$,
$F_{(1-\xi)N}(s) \leq m_N(S)$ for every state $s \in S$, as soon as
$N \geq N_0$. It follows that $\tau_0$ coincides with
$\tau_{truth}$ in the first $(1-\xi)N$ stages, so that with
probability at least $1-\xi$,
\[ \|\mu_{\sigma_{truth},\tau_0} - \mu_0\|_1 \leq |S|\xi. \]
This implies that
\[ U^1(\mu_{\sigma_{truth},\tau_0},y) > U^1(\mu_0,y) - (|S|+1)\xi. \]
For fixed $N$, as $\delta$ converges to 1, the discounted payoff
in each block converges to the average payoff in that block, and
therefore for $\delta$ sufficiently large
\[ \gamma^1_\delta(\sigma_{truth},\tau_0) > U^1(\mu_0,y) - (|S|+2)\xi. \]
Because $\sigma_0$ is a best reply to $\tau_0$, we deduce that
\[ \gamma^1_\delta(\sigma_0,\tau_0) \geq \gamma^1_\delta(\sigma_{truth},\tau_0) > U^1(\mu_0,y) - (|S|+2)\xi = U^1(\mu_0,y) - \eta. \]
We again use the fact that, for fixed $N$, as $\delta$ goes to 1,
the payoff $\gamma_\delta(\sigma_0,\tau_0)$ converges to
$U(\mu_{\sigma_0,\tau_0},y)$ to deduce that
\begin{equation}
\label{equ5.4}
U^1(\mu_{\sigma_0,\tau_0},y)>  U^1(\mu_0,y) -\eta.
\end{equation}

For fixed $N$, and for every $\delta$, the marginal distributions of $\mu_{\sigma_0,\tau_0}\in \Delta(S\times A)$
on  $S$ and $A$ are respectively equal to $m$ and to $m_N$.

Since the approximation $m^N$ converges to $m$ as $N\to +\infty$, the distribution $\mu_{\sigma_0,\tau_0}$ converges to the
set $\calM$ of copulas. Using Lemma \ref{lemm2}, Proposition \ref{prop eq} therefore follows from (\ref{equ5.4}).

\subsection{Proof of Lemma \ref{lemm2}}

Let a 
copula $\mu\in \calM$ be given.   Present $\mu$ as a
convex combination of the extreme points $(\mu_e)_e$ of $\calM$:
$\displaystyle\mu=\sum_{\mu_e\in \calM_e}\alpha_e \mu_e$, with
$\alpha_e\geq 0$ and $\displaystyle \sum_{\mu_e\in
\calM_e}\alpha_e=1$. Recall that $\mu_0$ is one of the extreme
points of $\calM$.

On the one hand, since $U^1$ is bi-linear, one has
\begin{eqnarray}
\nonumber
U^1(\mu_0,y)-U^1(\mu,y)&= &
\ep U^1(\mu_0,y_0)+(1-\ep)U^1(\mu_0,y_1)-\ep U^1(\mu,y_0)-(1-\ep)U^1(\mu, y_1)\\
\label{equ019}
&\geq& \ep\left(U^1(\mu_0,y_0)-U^1(\mu,y_0)\right)\\
&=& \ep\left((1-\alpha_{0})U^1(\mu_0,y_0)-\sum_{\mu_e\neq \mu_0}\alpha_e U^1(\mu_e,y_0)\right) \\
&\geq & \ep(1-\alpha_{0}) c_1,\label{eq10}
\end{eqnarray}
where the inequality (\ref{equ019}) holds because $y_1 \in
Y(\calM)$ and by \textbf{C1}.

On the other hand, one has
$\mu-\mu_0= \sum_{\mu_e\in \calM_e}\alpha_e (\mu_e-\mu_0)$,
hence
\begin{equation}
\label{eq11}
\|\mu-\mu_0\|_1\leq  c_2\sum_{\mu_e\in \calM_e,\mu_e\neq \mu_0}\alpha_e=c_2(1-\alpha_{0}).
\end{equation}
The result follows from (\ref{eq10}) and (\ref{eq11}).

\section{Complements to the proof of Theorem \ref{theorem2}}

We here prove Lemma \ref{lemm4ter}.  For clarity, we introduce yet another copy $T$ of the set $S$. 
Intuitively, fictituous states are $T$-valued, while realized ones are $S$-valued.

Define $\calM'\subseteq \Delta (S\times T)$ to be the set of distributions $\mu\in  \calM$ such that the following property \textbf{P}
holds:
\begin{description}
\item[Property P.] For every $(s,t)\in S\times T$, one has
\begin{equation}
\label{equ5} \sum_{s'\in S}\mu(s'\mid t)p(s\mid s')=\sum_{t'\in
T}\mu(s\mid t')p(t'\mid t).
\end{equation}
\end{description}

We will prove 
\begin{lemma}\label{lemm4bis}
Under \textbf{Assumption A}, the set $\calM'$ coincides with the set $\calM$.
\end{lemma}

\begin{lemma}\label{lemm4}
Let $\mu\in \calM'$ be given. There exists an  $S$-valued process%
\footnote{The process $(\textbf{t}_n)_n$ is possibly defined on a probability space which
is an enlargement of the one on which $(\textbf{s}_n)_n$ is defined.}
$(\textbf{t}_n)_n$,
such that:
\begin{description}
\item[P1] The law of the sequence $(\textbf{t}_n)_n$ is the same
as the law of the sequence $(\textbf{s}_n)_n$. \item[P2] The law
of the pair $(\textbf{s}_n,\textbf{t}_n)$ is $\mu$, for each stage
$n\in \dN$. \item[P3] The conditional law of $\textbf{s}_n$, given
$\textbf{t}_1,\ldots, \textbf{t} _n$ is $\mu(\cdot \mid
\textbf{t}_n)$. \item[P4] Conditional on $\textbf{s}_n$, the
vector $(\textbf{t}_1,\ldots, \textbf{t}_n)$ is independent of the
future states $(\textbf{s}_{n+1},\textbf{s}_{n+2},\ldots)$.
\end{description}
\end{lemma}

We emphasize that only Lemma \ref{lemm4bis} makes use of
\textbf{Assumption A}. This has the following consequence. Given
$\mu\in \calM'$, using Lemma \ref{lemm4} and the construction of the paper,
 one has $U^1(\mu_0,y)\geq U^1(\mu,y)$. Thus, the
conclusion $U^1(\mu_0,y)=\max_{\mu\in \calM'}U^1(\mu,y)$ holds,
irrespective of whether \textbf{Assumption A} is met or not.

\subsection{Proof of Lemma \ref{lemm4}}
Let $\mu\in \calM'$ be given, and define $\bar\mu\in \Delta(T\times S\times T)$ by
\begin{equation}
\label{equA0}
\bar\mu(t',s,t)=\mu(s,t) p(t\mid t') \frac{m(t')}{m(t)}, \mbox{ } (t',s,t)\in T\times S\times T.
\end{equation}
 For every two indices $i,j\in \{1,2,3\}$ with $i < j$, denote by $\bar \mu_{i,j}$ the marginal of $\bar\mu$ on the $i$-th and $j$-th coordinates.

 We will use the following properties of $\bar \mu$.

 \begin{lemma}\label{lemm3}
 One has
 \begin{enumerate}
 \item $\bar \mu_{2,3}=\mu$;
 \item $\bar \mu_{1,3}(t',t)=m(t')p(t\mid t')$ for every $t,t'\in T$;
 \item $\bar\mu_{1,2}(t',s')=\sum_{s\in S}\bar\mu_{2,3}(s,t')p(s'\mid s),$
 for each $t'\in T,s'\in S$;
 \item  $\bar\mu(s\mid t',t)=\mu(s\mid t)$ for each $(t',s,t)\in T\times S\times T$.
 \end{enumerate}
 \end{lemma}

\begin{proof}
We prove the four claims in turn. Let $s,t\in S\times T$ be given. One has
\begin{eqnarray*}
\bar\mu_{2,3}(s,t)&=& \sum_{t'\in T}\bar\mu(t',s,t)= \sum_{t'\in T}\mu(s,t)p(t\mid t')\frac{m(t')}{m(t)}\\
&=& \frac{\mu(s,t)}{m(t)}\sum_{t'\in T}p(t\mid t')m(t') = \mu(s,t),
\end{eqnarray*}
which proves the first claim.

To prove the second claim, let $t',t\in T$ be given. One has
\[
\bar\mu_{1,3}(t',t)= \sum_{s\in S}\bar\mu(t',s,t)
=  \sum_{s\in S}\mu(s,t)p(t\mid t')\frac{m(t')}{m(t)}= p(t\mid t') m(t'),
\]
where the last equality holds since the marginal distribution of $\mu$ on $S$ is $m$.

We turn to the third claim. Let $t'\in T$, $s'\in S$ be given.
By the first claim, and since $\mu \in \calM'$, one has
\begin{equation}
\label{equA1}
\sum_{s\in S}\bar\mu_{2,3} (s,t')p(s'\mid s) = \sum_{s\in S} \mu(s,t')p(s'\mid s) = m(t')\sum_{t\in T}\mu(s'\mid t) p(t\mid t').
\end{equation}

On the other hand,
\begin{equation}
\label{equA2}
\bar\mu_{1,2}(t',s')=\sum_{t\in T}\bar\mu(t',s',t)=\sum_{t\in T} \mu(s',t)p(t\mid t')\frac{m(t')}{m(t)}.
\end{equation}
The third claim follows from (\ref{equA1}) and (\ref{equA2}).

Finally, let $(t',s,t)\in T\times S\times T$ be given. By the second claim,
\[\bar\mu(s\mid t',t)=\frac{\bar\mu (t',s,t)}{\mu_{1,3}(t',t)}= \frac{\mu(s,t)p(t\mid t')}{p(t\mid t')m(t')}\times \frac{m(t')}{m(t)}=\mu(s\mid t),\]
and the fourth claim follows.
\end{proof}

\bigskip

We construct the sequence $(\textbf{t}_n)_n$ as follows. The initial values $\textbf{t}_0$ and $\textbf{t}_1$ are drawn according to
the conditional distribution $\bar \mu(\cdot |\textbf{s}_1)\in \Delta (T\times T)$.
For $n\neq 2$, $\textbf{t}_n$ is drawn according to the conditional distribution $\bar\mu(\cdot \mid \textbf{t}_{n-1},\textbf{s}_n)$.
In this construction, $\textbf{t}_0$ is used to unify the treatment of $\textbf{s}_1$ with that of $(\textbf{s}_n)_{n \geq 2}$.
Property \textbf{P4} thus holds by construction. Properties \textbf{P1} and  \textbf{P2} follow from the next lemma.

\begin{lemma} The law of $(\textbf{t}_{n-1},\textbf{s}_n,\textbf{t}_n)$ is equal to $\bar\mu$, for each stage $n\geq 1$.
\end{lemma}

\begin{proof}
We argue by induction. Observe that the law of $\textbf{s}_1$ is equal to $m$. Therefore,
\[\prob((\textbf{t}_0,\textbf{s}_1,\textbf{t}_1)=(t',s,t))= m(s)\bar\mu(t',t\mid s)=\bar\mu(t',s,t).\]
Assume that the claim holds for some $n \in \dN$. We will prove that the law of $(\textbf{t}_n,\textbf{s}_{n+1})$ is then
equal to $\bar\mu_{1,2}$. This follows from the following sequence of equalities, which holds for every $t'\in T,s\in S$:
\begin{eqnarray*}
\prob((\textbf{t}_{n},\textbf{s}_{n+1})=(t',s))&=& \sum_{s'\in S}\prob((\textbf{s}_{n},\textbf{t}_{n},\textbf{s}_{n+1})=(s',t',s))\\
&=& \sum_{s'\in S}\prob((\textbf{s}_{n},\textbf{t}_{n})=(s',t'))\times \prob(\textbf{s}_{n+1}=s|(\textbf{s}_{n},\textbf{t}_n)=(s',t))\\
&=& \sum_{s'\in S}\bar\mu_{2,3}(s',t')p(s\mid s')= \bar\mu_{1,2}(t',s),
\end{eqnarray*}

where the last equality follows from Lemma \ref{lemm3}(3) and \textbf{P4}.
Since the conditional law of $\textbf{t}_{n+1}$ given $(\textbf{t}_n,\textbf{s}_{n+1})$ is equal to $\bar\mu(\cdot\mid \textbf{t}_n,\textbf{s}_{n+1})$,
this yields the claim for $n+1$.
\end{proof}

\bigskip

Finally, property \textbf{P3} follows from the second part of the next lemma.
The first part of the lemma is needed to the proof of the second part.

\begin{lemma}
(1) The conditional law of $\textbf{t}_n$ given $(\textbf{t}_0,\ldots, \textbf{t}_{n-1})$ coincides with the conditional law of $\textbf{t}_n$ given
$\textbf{t}_{n-1}$.

(2) The conditional law of $\textbf{s}_n$ given $(\textbf{t}_0,\ldots, \textbf{t}_{n-1},\textbf{t}_n)$
coincides with the conditional law of $\textbf{s}_n$ given
$\textbf{t}_{n}$.
\end{lemma}

\begin{proof}
The proof is by induction.
For $n=1$, the first statement trivially holds, while the second statement holds by Lemma \ref{lemm3}(1). Assume that the claim holds for some $n \in \dN$.
For brevity, we denote by $t_{n}, s_n,\cdots$ generic values of $\textbf{t}_n,\textbf{s}_n\cdots$, and we write $\prob(t_n,s_n)$ instead
of $\prob((\textbf{t}_n,\textbf{s}_n)=(t_n,s_n))$.

Observe first that by the definition of $(\textbf{t}_n)$,
\begin{eqnarray}
\nonumber
\prob(t_{n+1}\mid t_0,\ldots, t_n)&=& \sum_{s_{n+1} \in S}\prob(s_{n+1}\mid t_0,\ldots, t_n)\prob(t_{n+1}\mid t_0,\ldots,t_n,s_{n+1}) \\
&=& \sum_{s_{n+1} \in S}\prob(s_{n+1}\mid t_0,\ldots, t_n)\times
\bar\mu(t_{n+1}\mid t_n,s_{n+1}). \label{equA3}
\end{eqnarray}

Moreover,
\begin{eqnarray}
\nonumber
\prob(s_{n+1}\mid t_0,\ldots, t_n)&=& \sum_{s_{n} \in S}\prob(s_{n}\mid t_0,\ldots, t_n)\prob(s_{n+1}\mid s_n,t_0,\ldots,t_n) \\
\nonumber
&=& \sum_{s_{n} \in S}\prob(s_{n}\mid t_0,\ldots, t_n)\times p(s_{n+1}\mid s_n)\\
&=& \sum_{s_{n} \in S}\prob(s_{n}\mid t_n)\times p(s_{n+1}\mid
s_n), \label{equA4}
\end{eqnarray}
where the last equality holds by the induction hypothesis.
Note that the right-hand side of (\ref{equA4}) is independent of $(t_0,t_1,\ldots,t_{n-1})$,
and therefore
\begin{equation}
\label{equ 17.1}
\prob(s_{n+1}\mid t_0,\ldots, t_n) = \prob(s_{n+1}\mid t_n).
\end{equation}

Plugging (\ref{equA4}) in (\ref{equA3}), one obtains
\begin{eqnarray*}
\prob(t_{n+1}\mid t_0,\ldots, t_n)&=& \sum_{s_{n+1} \in S}\sum_{s_n \in S}\prob(s_n\mid t_n)\times p(s_{n+1}\mid s_n)\times \bar\mu (t_{n+1}\mid t_n,s_{n+1}).
\end{eqnarray*}
The right hand side is independent of $t_1,\ldots, t_{n-1}$, hence it is equal to $\prob(t_{n+1}\mid t_n)$,
and the first part of the lemma follows.

We turn to the second statement. One has
\begin{eqnarray*}
\prob(s_{n+1}\mid t_0,\ldots, t_{n+1})&=& \frac{\prob(s_{n+1},t_{n+1}\mid t_0,\ldots,t_n)}{\prob(t_{n+1}\mid t_0,\ldots, t_n)}
= \frac{\prob(s_{n+1}\mid t_0,\ldots,t_n)\times \prob(t_{n+1}\mid s_{n+1},t_0,\ldots,t_n)}{\prob(t_{n+1}\mid t_0,\ldots, t_n)}\\
&=& \frac{\prob(s_{n+1}\mid t_n)\bar\mu(t_{n+1}\mid s_{n+1},t_n)}{\prob(t_{n+1}\mid t_n}
= \prob(s_{n+1}\mid t_n) \frac{\mu(s_{n+1},t_{n+1})}{\bar\mu(t_n,s_{n+1})} \frac{m(t_n)}{m(t_{n+1})}\\
&=& \frac{\mu(s_{n+1},t_{n+1})}{m(t_{n+1})}
= \prob(s_{n+1}\mid t_{n+1}),
\end{eqnarray*}
where the third equality holds by (\ref{equ 17.1}), the construction of $(\textbf{t}_n)_n$ and the first claim,
and the fourth equality holds by (\ref{equA0}).
This concludes the proof of the induction step.\end{proof}

The proof of Lemma \ref{lemm4} is now completed.

\subsection{Proof of Lemma \ref{lemm4bis}}

We here verify that if Assumption \textbf{A} holds then $\calM=\calM'$.
Let $p$ be a transition function such that $p(s'\mid s)=\alpha_{s'}$ for every two states
$s\neq s'$, and $p(s\mid s)=1-\displaystyle \sum_{s'\neq s}\alpha_{s'}$.
Set $C=\displaystyle \sum_{s\in S}\alpha_s$.
One can verify that the invariant measure of $p$ is given by $m(s)=\frac{\alpha_s}{C}$ for each $s\in S$.

Let $\mu\in \calM$. We will prove that for every $(t, s')\in T\times S$, the equality
\begin{equation}
\label{eq1}\sum_{s\in S} \mu(s\mid t) p({s}'\mid s)=\sum_{ t' \in T}  \mu({s}'\mid {t'}) p({t'}\mid t)\end{equation}
holds.
Fix $t \in T$ and ${s'} \in S$. Observe that
\begin{eqnarray}
\nonumber
 \sum_{s\in S} \mu(s\mid t) p({s}'\mid s)&=&\mu({s}'\mid t) \left(1-\sum_{s\neq s'} \alpha_{s}\right)  + \sum_{s \neq s'} \alpha_{s'} \mu(s\mid t)\\
\nonumber
 &=&   \mu(s'\mid t)\left(1-\sum_{s\neq s'} \alpha_{s}\right)  +   \alpha_{s'} (1-\mu(s'\mid t))\\
 &=&  \alpha_{s'} + \mu(s'\mid t) (1-C).
\label{equ5.2}
 \end{eqnarray}
On the other hand, one has
\begin{eqnarray}
\label{equ5.3}
\sum_{t' \in T}  \mu(s'\mid t') p(t'\mid t)
&=&\mu(s'\mid t) \left(1-\sum_{t'\neq t} \alpha_{t'}\right)  + \sum_{t'\neq t}\mu(s'\mid t')   \alpha_{t'}\\
&=&\mu(s'\mid t) \left(1-C+ \alpha_{t}\right)  + \sum_{t'\neq t}\mu(s'\mid t')   \alpha_{t'}.
\end{eqnarray}
When subtracting (\ref{equ5.3}) from (\ref{equ5.2}) one obtains
\begin{eqnarray}
\nonumber
\sum_{s\in S} \mu(s\mid t) p(s'\mid s)-\sum_{t' \in T}  \mu(s'\mid t') p(t'\mid t) &= &
 \alpha_{s'}  - \mu(s'\mid t) \alpha_t -\sum_{t'\neq t}\mu(s'\mid t')   \alpha_{t'} \\
\nonumber
 &=&
\alpha_{s'}-\sum_{t'\in S}\mu(s'\mid t')\alpha_{t'}\\
\label{equ 20.1}
&=&
\alpha_{s'}-  C\sum_{t'\in S}\mu(s'\mid t')m(t') \\
\nonumber
&=&
\alpha_{s'}-  C m(s') =
0,
\label{equ 20.3}
\end{eqnarray}
where (\ref{equ 20.1}) and (\ref{equ 20.3}) hold since $\alpha(s) = Cm(s)$ for every $s \in S$.
This proves (\ref{eq1}), as desired.

\section{Proof of Theorem \ref{theorem3}}

The proof of Theorem \ref{theorem3} consists of two independent parts.
We first prove that, if condition \textbf{B} does not hold for some game $G$, then it does not hold throughout some  neighborhood of $G$.

\begin{proposition}\label{prop1}
Let $G$ be a game that does not satisfy condition \textbf{B}.
Then there is a neighborhood $\calN$ of $G$ such that no game in $\calN$ satisfies condition \textbf{B}.
\end{proposition}

 \begin{proof}
The proof relies on the theory of semi-algebraic sets.
We refer to Bochnak, Coste and Roy (1998) for the results used below. Recall that the
set of extreme points of the polytope $\calM$ is denoted by $\calM_e$.

We will use the following two properties, that hold for constant
functions $y : S \to \Delta(B)$.
\begin{enumerate}
\item[\textbf{R1.}] If $y : S \to \Delta(B)$ is constant, then
$U^1(\mu_0,y) = U^1(\mu,y)$ for every $\mu \in \calM$.
\item[\textbf{R2.}] If $y : S \to \Delta(B)$ is constant, then
$U^2(\mu_0,y) \leq v^2$.
\end{enumerate}
Property \textbf{R1} holds because when $y$ is constant,
the payoff is independent of the sender's announcements. Property
\textbf{R2} holds because $v^2$ is the maximum of $U^2(\mu_0,y)$
over all constant functions $y$.

Given a payoff function $\tilde u: S\times B\to \dR^2$, we denote by $\calS(\tilde u)$ the system of inequalities
\[\tilde U^2(\mu_0,y)\geq v^2_{\tilde u} \mbox{ and } \tilde U^1(\mu_0,y)\geq \tilde U^1(\mu,y), \mbox{for all } \mu\in \calM_e,\]
with unknowns $y:S\to \Delta(B)$,
where $v^2_{\tilde u} = \max_{b \in B}\tilde U^2(\mu_0,b)$ is the min-max value of the receiver in the game with payoffs $\tilde u$.

We say that a vector $y\in \dR^{S\times B}$ is \emph{constant} if
$y(s,b)$ only depends on $b$.

Let $u$ denote the payoff function of $G$. By assumption, any solution $y$ to $\calS(u)$ is constant.
We will show that this implies that all solutions to $\calS(\tilde u)$ are constant, for all $\tilde u$ in a neighborhood of $u$.

Assume to the contrary that for every $\ep > 0$ there is a payoff function $u_\ep\in \dR^{2(S\times B)}$ such that (i) $\|u-u_\ep\|<\ep$, and (ii)
the system $\calS(u_\ep)$ has a non-constant solution $y_\ep\in \dR^{S\times B}$.

This implies that there is a \emph{semi-algebraic} map $\ep\in (0,1) \mapsto (u_\ep,y_\ep)$ such that
(i) $\lim_{\ep \to 0} u_\ep = u$, and (ii) $y_\ep$ is a \emph{non-constant} solution to $\calS(u_\ep)$  for every $\ep>0$ small enough.

In particular, the map $\ep\mapsto y_\ep$ has an expansion to a Puiseux series in a neighborhood of zero: there exist $\ep_0>0$,
a natural number $r$ and
 vectors $y_k\in \dR^{S\times B}$ for $k \geq 0$ such that
\[ y_\ep = \sum_{k=0}^\infty \ep^{\frac{k}{r}} y_k, \]
 for every $\ep\in (0,\ep_0)$, and a similar expansion exists for the map $\ep\mapsto u_\ep$.

Note that $y_0=\lim_{\ep\to 0} y_\ep$. This implies in particular that $y_0(\cdot, s)\in \Delta(B)$ for every $s \in S$, and that $y_0$ is a solution to
$\calS(u)$.
In particular,  $y_0$ is constant.

Because $y_\ep(\cdot, s) \in \Delta(B)$,
it follows that $\sum_{b\in B} y_\ep(b\mid s) = 1$ for every $\ep>0$ and every $s\in S$,
so that $\sum_{b\in B} y_k(b\mid s) = 0$ for every $k \geq 1$ and every $s\in S$.

Let $l \geq  0$ be the maximal integer such that $y_0,y_1,\ldots,y_l$ are constant functions.
Because $y_\ep$ is non constant for every $\ep>0$, we have $l < \infty$.
Define a vector $d \in \dR^B$ by
\[ d(b) = \min_{s \in S} y_{l+1}(b, s), \ \ \ \forall b \in B. \]
Note that
\begin{eqnarray}
y_\ep &=& \sum_{k=0}^\infty \ep^{\frac{k}{r}} y_k\\
&=& \left(\sum_{k=0}^l \ep^{\frac{k}{r}} y_k + \ep^{\frac{l+1}{r}}d\right) + \ep^{\frac{l+1}{r}}(y_{l+1}-d) + \sum_{k=l+2}^\infty \ep^{\frac{k}{r}} y_k.
\end{eqnarray}
The first term $\left(\sum_{k=0}^l \ep^{\frac{k}{r}} y_k + \ep^{\frac{l+1}{r}}d\right)$ is independent of $s$,
and all its coordinates are non-negative because $y_\ep$ is non-negative for every $\ep > 0$.
Set
\[ z_\ep = \frac{\sum_{k=0}^l \ep^{\frac{k}{r}} y_k + \ep^{\frac{l+1}{r}}d}{1 + \ep^{\frac{l+1}{r}}\sum_{b \in B} d(b)} \in \dR^{S \times B}. \]
Then $z_\ep(\cdot, s) \in \Delta(B)$ for every $s \in S$, and $z_\ep$ is independent of $s$.
Set
\[ w(\cdot, s) = \frac{y_{l+1}(\cdot,s) - d}{-\sum_{b \in B} d(b)} \in \dR^{S \times B}, \ \ \ \forall s \in S. \]
Then $w(s) \in \Delta(B)$ and $w$ is non-constant.
We will show that $w$ solves $\calS(u)$, contradicting the assumption that all solutions of $\calS(u)$ are constant.

By \textbf{R2}, for every $\ep > 0$ we have $\tilde
U^2(\mu_0,z_\ep) \leq v^2_{u_\ep}$. But $\tilde
U^2(\mu_0,y_\ep) \geq v^2_{u_\ep}$, and $y_\ep$ is a convex
combination of $z_\ep$, $w$, and a ``tail'' which is of a lower
order of $\ep$; by taking the limit $\ep \to 0$ and using $v^2_{u_\ep} \to v^2$ we
obtain $U^2(\mu_0,w) \geq v^2$.

Fix $\mu \in \calM_e$.
By \textbf{R1} it follows that $\tilde U^1(\mu_0,z_\ep) = \tilde U^1(\mu,z_\ep)$.
Because $\tilde U^1(\mu_0,y_\ep) \geq \tilde U^1(\mu,y_\ep)$,
it follows for the same reasoning as above that $U^1(\mu_0,w) \geq U^1(\mu,w)$.
\end{proof}
\bigskip

We turn to the second part of the proof.
\begin{proposition}\label{prop2}
Let $G$ be a game such that condition \textbf{B} holds.
Then any neighborhood of the game $G$ contains a game $G'$ such that $\widehat{E}_{G'}(\calM)\neq \emptyset$.
\end{proposition}

\begin{proof}
The proof combines three independent lemmas. We first show that
there are perturbations of $u^2$ such that the inequality in (i)
holds strictly for the perturbed game. Next, we show that the map
$y$ may be assumed to be one-to-one. Finally, we construct
perturbations of $u^1$ such that the inequalities in (ii) will be
strict.

\begin{lemma}\label{lemm31}
Let $G$ be a game with payoff function $u$, and let $y:S\to
\Delta(B)$ be a non-constant function such that $U^2(\mu_0,y)\geq
v^2$. Then, any neighborhood of $u^2$ contains payoff functions
$\tilde u^2$ such that $\tilde U^2(\mu_0,y)>\tilde v^2$.

\end{lemma}

\begin{proof}
Define $P\in \Delta(S \times B)$ by $P(s,b):=m(s) y(b\mid s)$, for
$s\in S,b\in B$, and let $\ep>0$ be given. We abuse notations and
still denote by $P$ the two marginals of $P$ over $S$ and
$B$. Note that $P(s)=m(s)>0$ for each $s\in S$. Define $\tilde
u^2:S\times B\to \dR$ by $\tilde u^2(s,b)=u^2(s,b)$ if $P(b)=0$,
and
 $$\tilde{u}^2(s,b)=u^2(s,b) + \varepsilon \frac{P(s,b)}{P(s)P(b)}\mbox { if }P(b)>0.$$

We claim that $\tilde U^2(\mu_0,y)>\tilde v^2$. Since $\ep$ is
arbitrary, the result will follow. Note first that, for $b\in B$
such that $P(b)>0$, one has
 \[
  \tilde{U}^2(\mu_0,b) = U^2(\mu_0,b) + \varepsilon \sum_{s\in S} m(s) \frac{P(s,b)}{m(s)P(b)}=U^2(b, \mu_0) + \varepsilon.
    \]
Hence,  $\tilde{v}^2=v^2+ \varepsilon$ (see Eq. (\ref{equ minmax})).
On the other hand, since
$y(b\mid s) = \frac{P(s,b)}{m(s)} = \frac{P(s,b)}{P(s)} = P(b \mid
s)$,
  \begin{eqnarray*}
  \tilde{U}^2(\mu_0,y) & =  & U^2(\mu_0,y) + \varepsilon \sum_{s \in S,b \in B} m(s)  y(b\mid s)  \frac{P(s,b)}{P(s)P(b)}\\
  & = & U^2(\mu_0,y) + \varepsilon \sum_{s\in S} m(s) \sum_{b\in B}  \frac{P(b\mid s)^2}{P(b)}.
    \end{eqnarray*}

Viewed as a function of the probability distribution $q\in \Delta(B)$, the expression $\displaystyle \sum_{b\in B}\frac{(q(b))^2}{P(b)}$
is strictly convex, and admits a unique minimum equal to 1, when $q=P$.
Thus, for fixed state $s \in S$, one has $\sum_{b\in B}  \frac{P(b\mid s)^2}{P(b)}\geq
1$, with a strict inequality whenever the conditional distribution
$P(\cdot\mid s)$ differs from $P$. Since $y$ is non-constant, there
exist one state $s$ such that $P(\cdot\mid s)\neq P$. Therefore,
\[
\tilde{U}^2(\mu_0,y) > U^2(\mu_0,y)+ \varepsilon \geq v^2+\ep=\tilde{v}^2, \]
as desired.
\end{proof}

\bigskip

\begin{lemma}\label{lemm41}
Let $G$ be a game with payoff function $u$, and let $y:S\to \Delta(B)$ be such that $U^1(\mu_0,y)\geq U^1(\mu,y)$ for each $\mu\in \calM$.
Then, any neighborhood of $y$ in $\dR^{S\times B}$ contains a one-to-one function
$\tilde y:S\to \Delta(B)$ such that $U^1(\mu_0,\tilde y)\geq U^1(\mu,\tilde y)$ for each $\mu\in \calM$.
\end{lemma}

\begin{proof}
It suffices to show the existence of a one-to-one map
$\tilde z:S\to\Delta(B)$ such that $U^1(\mu_0,\tilde z)\geq U^1(\mu,\tilde z)$ for each $\mu\in \calM$.
Indeed, the conclusion of the lemma then follows by setting $\tilde y=(1-\ep)y+\ep \tilde z$, for $\ep>0$ small enough.

Let $(z_s)_{s\in S}$ be arbitrary distinct elements of  $\Delta(B)$.
Let $\tilde \phi$ be a permutation over $S$ that maximizes the sum $\displaystyle \sum_{s\in S}u^1(s,z_{\psi(s)})$ over all permutations $\psi$,
and set $\tilde z_s=z_{\tilde \phi(s)}$. By construction, one has
\[\sum_{s\in S}u^1(s,\tilde z_s)\geq \sum_{s\in S}u^1(s,\tilde z_{\phi(s)}),\]
for every permutation $\phi$ over $S$. By Lemma \ref{lemmbasic}, this implies $U^1(\mu_0,\tilde z)\geq U^1(\mu,\tilde z)$ for every $\mu\in \calM$,
as desired.
\end{proof}

\bigskip

\begin{lemma}\label{lemm5}
Let $G$ be a game with payoff function $u$,
and let $y:S\to \Delta(B)$ be a one-to-one map such that $U^1(\mu_0,y)\geq U^1(\mu,y)$ for each $\mu\in \calM$.
Then, any neighborhood of $u^1$ contains payoff functions $\tilde u^1$ such that $\tilde U^1(\mu_0,y)> \tilde U^1(\mu,y)$ for each $\mu\in \calM$,
$\mu\neq \mu_0$.
\end{lemma}
Note that the existence of a stationary strategy $y$ that satisfies the requirements follows from Lemma \ref{lemm41}.

\begin{proof}
Let $G$, $u$ and $y$ be as stated. Given $\ep>0$, we define $\tilde u^1:S\times B\to \dR$ by
$$\tilde{u}^1(s,b)=u^1(s,b) + \varepsilon  y(b\mid s).$$ We will prove that for every $\ep>0$,
one has
$\tilde U^1(\mu_0,y)> \tilde U^1(\mu,y)$ for each $\mu \in \calM \setminus \{ \mu_0\}$.

Given a permutation $\phi$ over $S$, we denote by $Y_\phi\in \dR^{S\times B}$ the vector whose $(s,b)$-component is equal to
$y(b\mid \phi(s))$.
Then,
\begin{eqnarray}\label{eq5}
\sum_{s\in S}\tilde u^1(s,y(\cdot\mid\phi(s)))
&=& \sum_{s\in S,b \in B}y(b \mid \phi(s)) \tilde u(s,b)\\
&=&\sum_{s\in S,b \in B}y(b \mid \phi(s)) u(s,b) + \ep\sum_{s\in S,b \in B}y(b \mid \phi(s))y(b \mid s)\\
&=&\sum_{s\in S}u^1(s,y(\cdot\mid\phi(s)) +\ep \langle Y_{Id},Y_{\phi}\rangle,
\end{eqnarray}
where $\langle Y_{Id},Y_{\phi}\rangle = \sum_{s\in S,b \in B}y(b \mid \phi(s))y(b \mid s)$ is the standard scalar product in $\dR^{S\times B}$.

Since $y$ is one-to-one, the vectors $Y_\phi$  and  $Y_{Id}$ are not co-linear as soon as $\phi \neq Id$. By Cauchy-Schwarz inequality, it follows that
\begin{equation}\label{eq4}\langle Y_{Id},Y_{\phi}\rangle
< \|Y_{Id}\|_2 \|Y_{\phi}\|_2 = \|Y_{Id}\|^2=\langle Y_{Id},Y_{Id}\rangle\end{equation}
where the first equality holds since the components of $Y_{\phi}$ are obtained by permuting the components of $Y_{Id}$.

On the other hand, observe that by Lemma \ref{lemmbasic}, one has
\begin{equation}
\label{equ25}
\sum_{s\in S}u^1(s,y(\cdot\mid\phi(s))\leq \sum_{s\in S}u^1(s,y(\cdot\mid s)).
\end{equation}
Plugging (\ref{equ25}) into (\ref{eq5}), one obtains
\[\sum_{s\in S}\tilde u^1(s,y(\cdot\mid\phi(s))<\sum_{s\in S}u^1(s,y(\cdot\mid s) + \langle Y_{Id},Y_{Id}\rangle = \sum_{s\in S}\tilde u^1(s,y(\cdot\mid s)).\]
By Lemma \ref{lemmbasic} this yields $\tilde U^1(\mu_0,y)>\tilde U^1(\mu,y)$, for every $\mu\neq \mu_0$ in $\calM$, as desired.
\end{proof}

\bigskip

The proof of Proposition \ref{prop2} follows from Lemmas \ref{lemm31}, \ref{lemm41} and \ref{lemm5}.
\end{proof}

\end{appendix}

\end{document}